\documentclass[11pt]{amsart}

\usepackage[T1]{fontenc}
\usepackage[margin=1.08in]{geometry}
\usepackage{amsmath,amssymb,amsthm,mathtools}
\usepackage{enumitem}
\usepackage[colorlinks=true,citecolor=blue,linkcolor=blue,urlcolor=blue]{hyperref}
\usepackage{amsfonts}
\usepackage{graphicx}
\usepackage{subfigure}
\usepackage{float}
\usepackage{cleveref}
\usepackage{tikz}
\usepackage{tikz-3dplot}
\usepackage{mathrsfs}
\usepackage{bm}
\usepackage{esint}
\usepackage{appendix}

\usepackage{booktabs,longtable,array}

\newtheorem{theorem}{Theorem}[section]
\newtheorem{proposition}[theorem]{Proposition}
\newtheorem{lemma}[theorem]{Lemma}
\newtheorem{corollary}[theorem]{Corollary}

\theoremstyle{definition}
\newtheorem{definition}[theorem]{Definition}
\theoremstyle{remark}
\newtheorem{remark}[theorem]{Remark}
\newcommand{\NS}{\operatorname{NS}}
\newcommand{\Nef}{\operatorname{Nef}}
\newcommand{\Aut}{\operatorname{Aut}}
\newcommand{\Amp}{\operatorname{Amp}}

\newcommand{\rk}{\operatorname{rk}}
\newcommand{\Kah}{\mathcal K}
\newcommand{\Pos}{\mathcal P}
\newcommand{\CY}{\operatorname{CY}}
\newcommand{\diam}{\operatorname{diam}}
\newcommand{\Vol}{\operatorname{Vol}}
\newcommand{\Ric}{\operatorname{Ric}}
\newcommand{\ddc}{dd^{c}}

\newcommand{\R}{\mathbb R}
\newcommand{\Q}{\mathbb Q}
\newcommand{\C}{\mathbb C}
\newcommand{\Z}{\mathbb Z}
\newcommand{\length}{\ell}

\newcommand{\Kthree}{\Lambda_{\mathrm{K3}}}
\newcommand{\cP}{\mathcal P}

\title{Rigid and Nonrigid Isotropic Nef Classes on K3 Surfaces}
\author{Xin Fu}
\address{School of Science, Institute for Theoretical Sciences, Westlake University, Hangzhou 310030, China}
\email{fuxin54@westlake.edu.cn}
\author{Zexuan Ouyang}
\address{Westlake Institute for Advanced Study, Westlake University, Hangzhou 310030, China}
\email{ouyangzexuan@westlake.edu.cn}

\thanks{Xin Fu is supported by the National Key R\&D Program of China (2024YFA1014800) and NSFC grant No. 12401073.}

\thanks{Zexuan Ouyang is supported by the National Key R\&D Program
	of China No. 2023YFA1009900.}

\date{}

\begin{document}
	
	\begin{abstract}
		Let $X$ be a fixed complex K3 surface and let
		$0\ne\alpha\in\overline{\Kah_X}$ be a $(1,1)$-class on the boundary of the K\"ahler cone with $\alpha^2=0$. We build a novel connection between two seemingly unrelated, intensively studied problems: the rigidity of a closed positive current in the class $\alpha$ and the collapse geometry of a sequence of Ricci-flat K\"ahler metrics in the K\"ahler class $\alpha+t\kappa$ as $t\to 0^+$. More precisely, we prove that
		$\alpha$ contains a unique closed positive $(1,1)$-current
		if and only if the Ricci-flat metrics in class $\alpha+t\kappa$ collapse to a point in the Gromov--Hausdorff sense.
		
		Then we exhibit the first example of a nonrigid irrational nef class $\alpha$ on a Kummer surface, which answers a question of
		Filip--Tosatti and Sibony--Soldatenkov--Verbitsky. We also construct many new examples of rigid nef classes on hyperk\"ahler manifolds, which extend the work of Filip--Tosatti and Sibony--Soldatenkov--Verbitsky.

		

	\end{abstract}
	
	\maketitle
	
	\setcounter{tocdepth}{1}
	\tableofcontents

	\section{Introduction}
	
	Let $X$ be a Calabi--Yau manifold and let
	$\Kah_X\subset H^{1,1}(X,\R)$ be its K\"ahler cone.  For
	$\beta\in\Kah_X$, the Calabi--Yau theorem \cite{Yau} gives a unique
	Ricci-flat K\"ahler form $\CY(\beta)$ in the class $\beta$.
	The behavior of these metrics as their K\"ahler classes approach the
	boundary of $\Kah_X$ has been studied extensively from both analytic
	and geometric perspectives.
	
	From the viewpoint of pluripotential theory, one seeks uniform estimates
	for the potentials of $\CY(\alpha+t\kappa)$ as $t\to0^+$, where
	$\alpha\in\partial\Kah_X$ and $\kappa\in\Kah_X$.  Closely related
	questions concern the positive currents representing $\alpha$.
	A class is called \emph{rigid} if it contains a unique closed positive
	$(1,1)$-current.  The boundedness of potentials and the rigidity of
	boundary classes are central questions in this setting, see the survey of Tosatti
	\cite{Tosatti}.
	
	From the geometric viewpoint, one studies the limits and regularity of
	the associated Ricci-flat metrics.  When the limiting class is big and
	nef, the metrics converge smoothly away from its null locus to a
	Ricci-flat metric, possibly incomplete
	\cite{TosattiLimits,CollinsTosatti}.  When the limiting class is nef but
	not big, the volumes tend to zero and collapsing phenomena arise.
	Cheeger--Tian developed analytic tools for collapsing Einstein
	four-manifolds \cite{CheegerTian}, and Sun--Zhang established a fibration structure in regular region based on the results of 	Cheeger--Fukaya--Gromov \cite{CGF}, and gave a rough classification for the collapsing limits \cite{SunZhang}. The classification is completed in \cite{Ouyang25}. The marked period asymptotics of Ouyang--Tian
	\cite{OuyangTian} provide further information used in this paper.
	In higher dimensions, a particularly important case occurs when the
	limiting class is semiample and induces a holomorphic fibration.
	Convergence, higher-order asymptotics, and the geometry of the resulting
	metric spaces have been studied in
	\cite{MR4201795,GrossWilson,SongTian07,FGS20,HSVZ,HeinTosattiSmooth,GrossTosattiZhangGH,GGZ,GG,Ouyang25,ZhaZha,MR2652468,MR4728691}.
	
	Our main geometric setting is a fixed complex K3 surface $X$ and a
	nonzero nef class $\alpha\in\overline{\Kah_X}$ satisfying
	$\alpha^2=0$.  We call such a class \emph{isotropic}, and its ray is
	\emph{irrational} if
	\[
	\R\alpha\cap H^2(X,\Q)=\{0\}.
	\]
	For an irrational class, a holomorphic fibration need not be available,
	and the collapsing geometry is less directly accessible.  We study
	two related questions: whether $\alpha$ is rigid, and whether the
	diameters of approaching Ricci-flat metrics tend to zero.
	Rigidity in this setting has been investigated in
	\cite{SSV,FilipTosattiSmooth,FilipTosattiKummer,FilipTosatti}.  The corresponding Ricci-flat
	metrics have uniformly bounded diameter \cite{Tosatti}, the issue is
	therefore to determine when the limit is a point and how this behavior
	is reflected in the positive currents representing $\alpha$.
	
	Our first result identifies current rigidity with metric collapse.
	
	\begin{theorem}
		\label{thm:equivalence}
		Let $X$ be a K3 surface and let
		$0\ne\alpha\in\overline{\Kah_X}$ satisfy $\alpha^2=0$.
		Put $D(\beta):=\diam(X,\CY(\beta))$ for $\beta\in\Kah_X$.
		Then the following conditions are equivalent:
		\begin{enumerate}[label=\textup{(\roman*)}]
			\item $\alpha$ is rigid;
			\item $D(\beta_i)\to0$ for every sequence
			$\beta_i\in\Kah_X$ satisfying $\beta_i\to\alpha$;
			\item $D(\alpha+t\kappa)\to0$ as $t\to0^+$ for some
			$\kappa\in\Kah_X$.
		\end{enumerate}
	\end{theorem}
	
	The implication \textup{(iii)}$\Rightarrow$\textup{(i)} follows from
	a quantitative estimate for potentials of positive currents, and  \textup{(ii)}$\Rightarrow$\textup{(iii)} is obvious. 
	For \textup{(i)}$\Rightarrow$\textup{(ii)}, if the diameters did not converge to zero,the regular fibration structures would produce
	two distinct positive currents in $\alpha$, contradicting rigidity.

	Our second result uses the fixed marking, which is not visible in an
	unmarked Gromov--Hausdorff compactification.
	
	\begin{theorem}
		\label{thm:dimension-exclusion}
		Let $X$ and $\alpha$ be as in Theorem~\ref{thm:equivalence}, and assume
		that $\R\alpha\cap H^2(X,\Q)=\{0\}$.
		Let $\beta_i\in\Kah_X$ converge to $\alpha$, put
		$g_i=g_{\CY(\beta_i)}$, and suppose $\inf_i\diam(X,g_i)>0$.
		Then the Gromov--Hausdorff limit of $(X,g_i)$ is an 1-dimensional intervel.

	\end{theorem}
	
	In particular, the Gromov--Hausdorff limit of this sequence can not
	has Hausdorff dimension two or three. 
	
	For radial degenerations, Theorem~\ref{thm:dimension-exclusion} yield the
	following dichotomy.
	
	\begin{corollary}
		\label{cor:dichotomy}
		Suppose that the ray of $\alpha$ is irrational, and let
		$\kappa\in\Kah_X$.
		If $\alpha$ is rigid, then
		\[
		\diam\bigl(X,\CY(\alpha+t\kappa)\bigr)\longrightarrow0
		\qquad\text{as }t\to0^+.
		\]
		If $\alpha$ is nonrigid, the diameters are bounded away from zero,
		and every sequence $t_i\to0^+$ has a subsequence along which
		$(X,\CY(\alpha+t_i\kappa))$ converges in the Gromov--Hausdorff sense
		to a one-dimensional interval.
	\end{corollary}
	
	Tosatti conjectured that every nef isotropic K3 class with irrational
	ray is rigid and that all approaching Ricci-flat metrics collapse
	to a point \cite[Conjectures~3.9 and~4.2]{Tosatti}.
	Theorem~\ref{thm:equivalence} shows that these conjectures are
	equivalent.  They are, however, false in this generality.
	
	\begin{theorem}
		\label{thm:kummer}
		There exists a Kummer K3 surface $X$ carrying a nonzero nef class
		$\alpha\in\partial\Kah_X$ such that
		\[
		\alpha^2=0,\qquad
		\R\alpha\cap H^2(X,\Q)=\{0\},
		\]
		with the following properties:
		\begin{enumerate}[label=\textup{(\alph*)}]
			\item $\alpha$ is nonrigid;
			\item there is a constant $c_0>0$ such that
			\[
			\diam\bigl(X,\CY(\alpha+\kappa)\bigr)\ge c_0
			\qquad\text{for every }\kappa\in\Kah_X;
			\]
			\item for every sequence $\kappa_i\in\Kah_X$ satisfying
			$\kappa_i\to0$, put $\omega_i=\CY(\alpha+\kappa_i)$ and
			$D_i=\diam(X,\omega_i)$.  Then
			\[
			\left(X,D_i^{-2}\omega_i\right)
			\xrightarrow[i\to\infty]{\mathrm{GH}}[0,1].
			\]
		\end{enumerate}
	\end{theorem}
	
	The lower bound follows from an explicit spectral estimate uniform
	over the whole K\"ahler cone of perturbations.  These examples raise
	the question of whether nonrigid irrational isotropic nef classes
	admit a systematic construction or classification, even on Kummer
	surfaces.
	
	We also construct rigid classes in the presence of walls of the
	K\"ahler cone.  Motivated by the arithmetic-dynamical arguments of
	Filip--Tosatti \cite{FilipTosatti} and
	Sibony--Soldatenkov--Verbitsky \cite{SSV}, we first work on an
	algebraic nef face of a projective K3 surface. Following the approach of \cite{FilipTosatti}, we prove the following
	criterion permits a prescribed negative-definite lattice of walls.
	
	\begin{theorem}
		\label{thm:newrigid}
		Let $X$ be a K3 surface with a primitive embedding
		$S:=N\oplus L\hookrightarrow\Kthree$ and an orthogonal decomposition
		\[
		\NS(X)_\R=N_\R\oplus L_\R,
		\]
		where $q|_{N_\R}$ is negative definite and $L_\R$ has signature
		$(1,n)$ for some $n\ge2$.  Let $\C_L$ be a connected component of
		$\{x\in L_\R:x^2>0\}$.  Suppose that:
		\begin{enumerate}[label=\textup{(\Alph*)}]
			\item $\Nef(X)\cap N_\R^\perp=\overline{\C_L}$;
			\item there is a subgroup $\Gamma<\Aut(X)$ which fixes $N_\R$
			pointwise and whose image in
			\[
			SO^+(L_\R):=\{g\in SO(L,\mathbb R):g(\C_L)=\C_L\}
			\]
			is a uniform lattice.
		\end{enumerate}
		Then every nonzero irrational isotropic class
		$\alpha\in\overline{\C_L}$ is rigid, and its unique positive current
		has bounded potential.
	\end{theorem}
	
	Finally, in Section~\ref{sec:k3-lattice-realization}, we use lattice constructions to realize its hypotheses with nontrivial ADE
	configurations and yield new explicit families of K3 surfaces carrying
	rigid classes on these nef faces.

	\section*{Declaration of AI use}
	During the preparation of this manuscript, the authors used OpenAI's ChatGPT to assist in filling in technical details and polishing the exposition. The main questions and theorems of this paper were formulated by the authors. AI was used to assist with the proofs of Proposition~\ref{prop:general-green-rigidity} and Lemma~\ref{lem:arithmetic}. The idea of using an anisotropic lattice, which is essential for our construction, in Section~\ref{sec:k3-lattice-realization} was due to AI.
	All AI-generated suggestions were reviewed, verified, and revised by the authors. 
	
	\section{Background and conventions}\label{sec:background}
	
	In this section, we recall the background material used in this paper and
	fix our conventions for lattices and cones. 
	
	\subsection{Quadratic lattices, roots}
	\label{subsec:lattice-background}
	
	We first recall standard terminology for quadratic lattices, following Nikulin~\cite{Nikulin}.
	
	\begin{definition}
		An \emph{integral lattice} $K$ is a free abelian group of finite rank
		equipped with a nondegenerate symmetric bilinear form
		$$
		(\ ,\ )_K:K\times K\longrightarrow\Z.
		$$
		The lattice is \emph{even} if $(x,x)_K\in2\Z$ for all $x\in K$.
		The signature of $K$ is the signature of the real bilinear form on
		$K_\R=K\otimes\R$.  A real quadratic space of
		signature $(1,n)$ is called \emph{Lorentzian}.
	\end{definition}
	
	\begin{definition}
		The dual lattice is
		$$
		K^\vee=\{x\in K_\Q:(x,K)\subset\Z\},
		$$
		and the finite group $A_K=K^\vee/K$ is the \emph{discriminant group}.
		If $K$ is even, its quadratic form descends to a finite quadratic form
		$$q_K:A_K\to\Q/2\Z.$$  The lattice is \emph{unimodular} when $A_K=0$.
		An embedding $K\hookrightarrow\Lambda$ is \emph{primitive} when
		$\Lambda/K$ is torsion free.  Finally,
		$$
		O(K)_{\mathrm{disc}}
		:=\ker\bigl(O(K)\longrightarrow O(A_K,q_K)\bigr)
		$$
		is the \emph{discriminant kernel}.  
	\end{definition}

	We now introduce the concepts of a root of an even lattice and an ADE lattice, which arise in several areas, including the theory of canonical singularities in birational geometry and the theory of simple Lie algebras.
	\begin{definition}\label{ADE}
		A \emph{root} of an even lattice $K$ is a vector
		$\delta\in K$ with $\delta^2=-2$. Its reflection is
		$$
		s_\delta(x)
		=x-\frac{2(x,\delta)}{\delta^2}\delta
		=x+(x,\delta)\delta.
		$$
		In particular, it fixes $\delta^\perp$ pointwise and sends $\delta$ to
		$-\delta$.
		An \emph{ADE lattice} in this paper means a nonzero orthogonal direct
		sum of the negative-definite root lattices
		$$
		A_n(-1)\ (n\ge1),\quad D_n(-1)\ (n\ge4),\quad
		E_6(-1),\ E_7(-1),\ E_8(-1).
		$$
		Equivalently, after choosing simple roots, its Gram matrix is the negative
		of a finite ADE Cartan matrix.  An \emph{ADE configuration} on a K3 surface
		is a collection of smooth rational curves with this intersection matrix.
	\end{definition}
	
	\subsection{K3 surfaces: lattices, Weyl chambers, periods, and Torelli}
	
	We next recall the standard lattice and cone theory of K3 surfaces. A \emph{K3 surface} is a compact complex surface $X$ with
	$K_X\simeq\mathcal O_X$ and $H^1(X,\mathcal O_X)=0$.  Its second integral cohomology, equipped with the cup
	product, is the even unimodular \emph{K3 lattice}
	$$
	H^2(X,\Z)\simeq
	\Kthree:=U^{\oplus3}\oplus E_8(-1)^{\oplus2},
	\qquad \operatorname{sign}(\Kthree)=(3,19),
	$$
	where $U$ is the hyperbolic plane with Gram matrix
	$\left(\begin{smallmatrix}0&1\\1&0\end{smallmatrix}\right)$.
	A \emph{marking} is a lattice isometry
	$\varphi:H^2(X,\Z)\xrightarrow{\sim}\Kthree$.
	The N\'eron-Severi and transcendental lattices are
	$$
	\NS(X)=H^{1,1}(X)\cap H^2(X,\Z),\qquad
	T(X)=\NS(X)^\perp.
	$$
	If $X$ is projective, the Hodge index theorem gives
	$\operatorname{sign}\NS(X)=(1,\rho(X)-1)$.  Conversely, a K3 surface is
	projective if and only if $\NS(X)$ contains a class of positive square.
	See \cite[Chapters~1--3]{HuybrechtsK3}.
	
	Let $X$ be projective.  Its \emph{positive cone} $\Pos(X)$ is the
	component of
	\begin{equation}\label{eqn:positivecone}
		\{x\in\NS(X)_\R:x^2>0\}
	\end{equation}
	containing the ample cone.  The ample cone $\Amp(X)$ consists of
	classes of ample $\R$-divisors, and its closure is the algebraic nef cone
	$\Nef(X)$.  In $H^{1,1}(X,\R)$, the K\"ahler cone
	$\Kah_X$ is the open cone of K\"ahler classes.  For a projective K3
	surface, $\Amp(X)=\Kah_X\cap\NS(X)_\R$.

	\begin{theorem}\label{thm:weyl}
		Let $X$ be a projective K3 surface and
		$\Delta_X=\{\delta\in\NS(X):\delta^2=-2\}$.
		For $\delta\in\Delta_X$, let $s_\delta$ be the reflection introduced in Definition~\ref{ADE}. 
		
		\begin{enumerate}
			\item For every $\delta\in\Delta_X$, Riemann--Roch implies that
			$\delta$ or $-\delta$ is effective.
			\item The root hyperplanes $\delta^\perp$ decompose $\Pos(X)$ into
			chambers.  The Weyl group
			$$
			W_X=\langle s_\delta:\delta\in\Delta_X\rangle
			\subset O^+(\NS(X))
			$$
			acts simply transitively on these chambers.
			\item The ample cone is the chamber
			$$
			\Amp(X)=
			\{x\in\Pos(X):(x,C)>0
			\text{ for every smooth rational curve }C\subset X\}.
			$$
			Put
			$$\Nef(X):=\overline{\Amp(X)}.$$ Then it is a fundamental domain for the
			Weyl-group action on the positive cone.
		\end{enumerate}
	\end{theorem}
	
	We now recall the two period results used to realize the lattice data
	geometrically. The period domain of marked K3 surfaces is
	$$
	\Omega_{\Kthree}
	=
	\{[\sigma]\in\mathbb P(\Kthree\otimes\mathbb C):
	\sigma^2=0,\;(\sigma,\bar\sigma)>0\}.
	$$
	The following global Torelli theorem is one of the fundamental theorems in the theory of K3 surfaces~\cite{BHPV}.
	\begin{theorem}\label{thm:torelli}
		\begin{enumerate}
			\item The period map of marked K3 surfaces is surjective.
			\item Let $X,X'$ be K3 surfaces and let
			$\psi:H^2(X',\Z)\to H^2(X,\Z)$ be a Hodge isometry.
			If $\psi$ maps a K\"ahler class of $X'$ to a K\"ahler class of
			$X$ (equivalently, it maps the K\"ahler cone onto the K\"ahler
			cone), then there is a unique isomorphism $f:X\to X'$ such that
			$f^*=\psi$.
		\end{enumerate}
	\end{theorem}
	We briefly discuss how we will use the Torelli theorem, Theorem~\ref{thm:torelli}, and Theorem~\ref{thm:weyl} when constructing new examples of rigid currents on K3 surfaces containing $(-2)$-curves.
	Item~(1) of the Torelli theorem allows us to find a K3 surface $X$ from a chosen period and also to prescribe $\NS(X)$ as the integral
	$(1,1)$-lattice orthogonal to that period. Item~(2) allows us to realize a particular lattice isometry by an automorphism.
	An arbitrary Hodge isometry need not itself be induced by an
	isomorphism of $X$ until its action on K\"ahler chambers has been controlled. For this,
	changing a marking by a root
	reflection fixes the period but changes the marked chamber and
	Theorem~\ref{thm:weyl} shows that one can reach exactly one chosen chamber.

	
	\section{A flat torus model}
	\label{sec:flat-torus-model}
	
	We begin with the elementary flat model motivating the rigidity problem.
	Let $A=V/\Lambda$ be a complex torus of complex dimension $n$, and let
	$\omega$ be a nonzero semipositive real $(1,1)$-form on $A$, invariant under
	translations and of complex rank $r<n$.  Set
	\[
	\alpha=[\omega],
	\qquad
	K=\ker\omega\subset V,
	\]
	and let $H\subset A$ be the closure of the image of $K$.  Thus $H$ is a
	connected real subtorus, possibly not a complex subtorus.
	We first record the metric behavior of these flat degenerations.
	
	\begin{lemma}
		\label{lem:flat-torus-degeneration}
		Set $W:=T_0H$.  The symmetric bilinear form
		\[
		g_\omega(u,v):=\omega(u,Jv),
		\qquad u,v\in V,
		\]
		is positive semidefinite with kernel $K\subset W$.  Let
		$\overline g_\omega$ denote the translation-invariant flat metric on $A/H$
		defined by the quotient norm
		\[
		|v+W|_{\overline g_\omega}
		:=
		\inf_{w\in W}|v+w|_{g_\omega},
		\qquad v\in V.
		\]
		Let $\beta_i$ be any sequence of translation-invariant K\"ahler forms on
		$A$ such that $[\beta_i]\to\alpha$.  Then
		\[
		(A,\beta_i)
		\xrightarrow[i\to\infty]{\mathrm{GH}}
		(A/H,\overline g_\omega),
		\qquad
		\dim_\R(A/H)=2n-\dim_\R H,
		\]
		and
		\[
		\diam(A,\beta_i)
		\longrightarrow
		\diam(A/H,\overline g_\omega).
		\]
		In particular, $\diam(A,\beta_i)\to0$ if and only if $H=A$.
	\end{lemma}
	
	\begin{proof}
		The natural map from the space of translation-invariant real $(1,1)$-forms
		on $A$ to $H^{1,1}(A,\R)$ is an isomorphism.  Hence
		$[\beta_i]\to[\omega]$ implies that $\beta_i\to\omega$ as forms.  The flat
		lattice distance formula gives the stated Gromov--Hausdorff limit, and
		continuity of diameter under Gromov--Hausdorff convergence gives the diameter
		limit.  Finally, the quotient metric has zero diameter exactly when $H=A$.
	\end{proof}
	
	\begin{proposition}
		\label{prop:torus}
		The following conditions are equivalent:
		\begin{enumerate}[label=\textup{(\roman*)}]
			\item $\alpha$ is rigid,
			\item for every sequence of translation-invariant K\"ahler forms
			$\beta_i$ on $A$ such that $[\beta_i]\to\alpha$,
			\[
			\diam(A,\beta_i)\longrightarrow0,
			\]
			\item for some sequence of translation-invariant K\"ahler forms
			$\beta_i$ on $A$ such that $[\beta_i]\to\alpha$,
			\[
			\diam(A,\beta_i)\longrightarrow0.
			\]
		\end{enumerate}
	\end{proposition}
	
	\begin{proof}
		We first show that $\alpha$ is rigid if and only if $H=A$.  Suppose that
		$H=A$.  Then every leaf of the kernel foliation is dense.  Let $T$ be a
		positive current cohomologous to $\omega$, and write
		$T=\omega+\ddc u$ with $u$ upper semicontinuous.  The argument of
		Sibony--Soldatenkov--Verbitsky \cite[Proposition~1.2]{SSV} shows that $T$
		vanishes in the directions tangent to the kernel foliation, so $u$ is
		pluriharmonic on each leaf.  If $p$ is a maximum point of $u$, the maximum
		principle implies that $u$ is constant on the leaf through $p$.  Since this
		leaf is dense, upper semicontinuity implies that $u$ is constant on $A$.
		Hence $T=\omega$, and $\alpha$ is rigid.
		
		Conversely, if $H\neq A$, choose a nonzero character
		$\ell\colon A\to\R/\Z$ vanishing on $H$, and let
		$\varphi=\cos(2\pi\ell)$.  Since $K$ is $J$-invariant,
		$\iota_v\ddc\varphi=0$ for $v\in K$.  Moreover,
		$\ddc\varphi\not\equiv0$, since otherwise $\varphi$ would be
		pluriharmonic and hence constant on $A$.  Thus, for sufficiently small
		$\varepsilon>0$, the forms
		\[
		\omega\pm\varepsilon\ddc\varphi
		\]
		are distinct smooth semipositive representatives of $\alpha$.  Hence
		$\alpha$ is nonrigid.
		
		Lemma~\ref{lem:flat-torus-degeneration} gives
		\textup{(i)}$\Rightarrow$\textup{(ii)}.  Since $\omega$ can be approximated
		by translation-invariant K\"ahler forms,
		\textup{(ii)}$\Rightarrow$\textup{(iii)} is immediate.  Finally,
		\textup{(iii)} and the same lemma give
		$\diam(A/H,\overline g_\omega)=0$.  Thus $H=A$, and $\alpha$ is rigid.
	\end{proof}
	
	When $A/H$ has positive dimension, pullbacks of suitable functions on the
	quotient yield transverse perturbations of $\omega$ within the class $\alpha$.
	The Kummer construction in Section~\ref{sec:kummer} realizes the same mechanism
	on a K3 surface.

	\section{Point collapse implies rigidity}
	\label{sec:point-collapse-rigidity}
	
	In this section, we prove that for any compact
	Calabi--Yau manifold, a limiting nef class is rigid whenever the Ricci-flat
	metrics associated with a single sequence of its K\"ahler perturbations
	collapse to a point.
	
	Throughout this section, let $X$ be a compact Calabi--Yau manifold of
	complex dimension $n$, with a fixed complex structure.  Let
	\[
	0\neq \alpha\in \overline{\mathcal K}_X
	\]
	be a nef class satisfying
	\[
	\alpha^n=0.
	\]
	
	We denote by
	\[
	\mathcal P(\alpha)
	:=
	\{T\geq 0:\ T \text{ is a closed real $(1,1)$-current and }
	[T]=\alpha\}
	\]
	the set of positive currents in the class $\alpha$.
	
	\subsection{A fixed normalized volume measure}
	\label{subsec:fixed-measure}
	
	Fix a nowhere vanishing holomorphic volume form $\Omega$ on $X$.  If
	$\omega$ is any Ricci-flat Kähler form on the fixed complex manifold $X$,
	then
	\[
	\operatorname{Ric}(\omega)
	=
	-dd^c\log
	\frac{\omega^n}{i^{n^2}\Omega\wedge\overline{\Omega}}
	=0.
	\]
	Hence
	\[
	\omega^n=c_\omega\,i^{n^2}\Omega\wedge\overline{\Omega}
	\]
	for some constant $c_\omega>0$.  It follows that the normalized volume
	measure
	\[
	d\mu
	:=
	\frac{\omega^n}{\int_X\omega^n}
	=
	\frac{i^{n^2}\Omega\wedge\overline{\Omega}}
	{\int_X i^{n^2}\Omega\wedge\overline{\Omega}}
	\]
	is independent of the Ricci-flat Kähler class.
	
	If $T_0,T_1\in\mathcal P(\alpha)$, the
	$\partial\bar\partial$-lemma for currents gives
	\[
	T_1-T_0=dd^c u
	\]
	for some real $u\in L^1(X)$, unique up to an additive constant.
	We normalize $u$ by
	\[
	\int_Xu\,d\mu=0
	\]
	and define
	\[
	d_\mu(T_0,T_1)
	:=
	\int_X|u|\,d\mu.
	\]
	In particular,
	\[
	d_\mu(T_0,T_1)=0
	\quad\Longleftrightarrow\quad
	T_0=T_1.
	\]
	
	\subsection{A Green function estimate}
	
	We first establish a quantitative estimate for an arbitrary K\"ahler class.
	
	\begin{proposition}
		\label{prop:general-green-rigidity}
		Let $\beta\in\Kah_X$.  Write $\omega=\CY(\beta)$ for the unique
		Ricci-flat K\"ahler form in the class $\beta$, and set
		$D_\beta:=\operatorname{diam}(X,\omega)$.  Then
		\[
		\operatorname{diam}_{d_\mu}\mathcal P(\alpha)
		\leq
		C_n D_\beta^2
		\frac{
			\int_X\alpha\wedge\beta^{n-1}
		}{
			\int_X\beta^n
		}.
		\]
		Here $C_n$ depends only on the dimension and on the normalization
		conventions for $d^c$ and the Laplacian.
	\end{proposition}
	
	\begin{proof}
		Fix $T_0,T_1\in\mathcal P(\alpha)$, and choose a real $u\in L^1(X)$
		such that
		\[
		T_1-T_0=dd^c u,
		\qquad
		\int_Xu\,d\mu=0.
		\]
		Let
		\[
		dV_\omega:=\frac{\omega^n}{n!},
		\qquad
		V_\omega:=\operatorname{Vol}(X,\omega).
		\]
		Since the normalized volume measure is independent of $\beta$,
		\[
		d\mu=\frac{dV_\omega}{V_\omega},
		\]
		and therefore
		\[
		\int_Xu\,dV_\omega=0.
		\]
		
		There is a constant $c_n>0$, depending only on $n$, such that, in the
		sense of distributions,
		\[
		(\Delta_\omega u)\,dV_\omega
		=
		c_n\,dd^cu\wedge\omega^{n-1}.
		\]
		Thus
		\[
		\nu
		:=
		(\Delta_\omega u)\,dV_\omega
		=
		c_n(T_1-T_0)\wedge\omega^{n-1}
		\]
		is a finite signed measure of total mass zero.  Since
		$T_j\wedge\omega^{n-1}$ are positive measures, the total variation of
		$\nu$ is at most
		\[
		c_n\int_X(T_0+T_1)\wedge\omega^{n-1}
		=
		2c_n\int_X\alpha\wedge\beta^{n-1}.
		\]
		
		Let $G_\omega(x,y)$ be the Green kernel of $\Delta_\omega$ with zero
		mean.  Since $\operatorname{Ric}(\omega)=0$, the standard Green function
		estimate \cite[Theorem 3.2]{BandoMabuchi} gives
		\[
		G_\omega(x,y)
		\geq
		-C_n\frac{D_\beta^2}{V_\omega}.
		\]
		Because
		\[
		\int_XG_\omega(x,y)\,dV_\omega(x)=0,
		\]
		the positive and negative parts of $G_\omega(\cdot,y)$ have equal
		integral.  Hence
		\[
		\sup_{y\in X}
		\int_X|G_\omega(x,y)|\,dV_\omega(x)
		\leq
		C_nD_\beta^2.
		\]
		
		Define the Green potential
		\[
		v(x)
		:=
		\int_XG_\omega(x,y)\,d\nu(y),
		\]
		with the sign chosen according to the convention for $\Delta_\omega$.
		By the preceding estimate and Fubini's theorem,
		\[
		\|v\|_{L^1(dV_\omega)}
		\leq
		C_nD_\beta^2
		\int_X\alpha\wedge\beta^{n-1}.
		\]
		Moreover,
		\[
		\Delta_\omega v=\Delta_\omega u
		\]
		in the sense of distributions.  Thus $u-v$ is distributionally
		harmonic.  Elliptic regularity implies that it is smooth and harmonic,
		hence constant.  Both $u$ and $v$ have zero $dV_\omega$-average, so
		$u=v$ almost everywhere.  Therefore,
		\[
		\int_X|u|\,dV_\omega
		\leq
		C_nD_\beta^2
		\int_X\alpha\wedge\beta^{n-1}.
		\]
		Dividing by
		\[
		V_\omega
		=
		\frac{1}{n!}\int_X\beta^n
		\]
		and absorbing the factorial into $C_n$ gives
		\[
		d_\mu(T_0,T_1)
		\leq
		C_nD_\beta^2
		\frac{
			\int_X\alpha\wedge\beta^{n-1}
		}{
			\int_X\beta^n
		}.
		\]
		Taking the supremum over $T_0,T_1\in\mathcal P(\alpha)$ proves the
		proposition.
	\end{proof}
	
	\begin{corollary}
		\label{cor:general-quantitative-rigidity}
		Let $\beta_i\to\alpha$ be K\"ahler classes and put
		\[
		\omega_i=\CY(\beta_i),
		\qquad
		D_i=\operatorname{diam}(X,\omega_i).
		\]
		If
		\[
		D_i^2
		\frac{
			\int_X\alpha\wedge\beta_i^{n-1}
		}{
			\int_X\beta_i^n
		}
		\longrightarrow0,
		\]
		then $\alpha$ is rigid.
	\end{corollary}
	
	\begin{proof}
		For any $T_0,T_1\in\mathcal P(\alpha)$,
		Proposition~\ref{prop:general-green-rigidity} gives
		\[
		d_\mu(T_0,T_1)
		\leq
		C_nD_i^2
		\frac{
			\int_X\alpha\wedge\beta_i^{n-1}
		}{
			\int_X\beta_i^n
		}.
		\]
		The left side is independent of $i$, while the right side
		tends to zero.  Hence $T_0=T_1$.
	\end{proof}
	
	\subsection{K\"ahler perturbations of the limiting class}
	
	For perturbations in a K\"ahler direction, the cohomological factor in
	Corollary~\ref{cor:general-quantitative-rigidity} has an elementary uniform
	bound.
	
	\begin{lemma}
		\label{lem:positive-approach-ratio}
		Let $\kappa\in\Kah_X$ and put $\beta=\alpha+\kappa$.  Then
		$\beta\in\Kah_X$ and
		\[
		0\leq
		\frac{
			\int_X\alpha\wedge\beta^{n-1}
		}{
			\int_X\beta^n
		}
		=
		1-
		\frac{
			\int_X\kappa\wedge\beta^{n-1}
		}{
			\int_X\beta^n
		}
		<1.
		\]
	\end{lemma}
	
	\begin{proof}
		The class $\beta$ is K\"ahler because $\alpha$ is nef and $\kappa$ is
		K\"ahler.  Positivity of the intersection of a nef class with a K\"ahler
		class gives the first inequality,
		while
		\[
		\int_X\beta^n
		=
		\int_X\alpha\wedge\beta^{n-1}
		+
		\int_X\kappa\wedge\beta^{n-1}
		\]
		and the second term is strictly positive.  This proves the identity and the
		strict upper bound.
	\end{proof}
	
	\begin{theorem}
		\label{thm:point-collapse-general-CY}
		Let $X$ be a compact Calabi--Yau manifold of complex dimension $n$, and let
		\[
		0\neq\alpha\in\overline{\mathcal K}_X,
		\qquad
		\alpha^n=0.
		\]
		Let $\kappa_i\to0$ be K\"ahler classes.  Put
		\[
		\beta_i=\alpha+\kappa_i,
		\qquad
		\omega_i=\CY(\beta_i),
		\qquad
		D_i=\operatorname{diam}(X,\omega_i).
		\]
		If $D_i\to0$, then $\alpha$ is rigid.
	\end{theorem}
	
	\begin{proof}
		Proposition~\ref{prop:general-green-rigidity} and Lemma
		\ref{lem:positive-approach-ratio} give
		\[
		\operatorname{diam}_{d_\mu}\mathcal P(\alpha)
		\leq C_nD_i^2
		\]
		for every $i$.  Letting $i\to\infty$ shows that this diameter is zero, so
		$\alpha$ is rigid.
	\end{proof}
	
	\begin{corollary}
		\label{cor:radial-green}
		Let $X$ be a compact Calabi--Yau manifold of complex dimension $n$, and let
		$0\neq\alpha\in\overline{\mathcal K}_X$ satisfy $\alpha^n=0$.  If
		$\alpha$ is nonrigid, then there is $c_\alpha>0$ such that, for every
		K\"ahler class $\kappa$,
		\[
		\diam\bigl(X,\CY(\alpha+\kappa)\bigr)\geq c_\alpha.
		\]
	\end{corollary}
	
	\begin{proof}
		Choose distinct $T_0,T_1\in\mathcal P(\alpha)$ and set
		$\delta=d_\mu(T_0,T_1)>0$.  Proposition
		\ref{prop:general-green-rigidity} and Lemma
		\ref{lem:positive-approach-ratio} give
		\[
		\delta
		\leq
		C_n\diam\bigl(X,\CY(\alpha+\kappa)\bigr)^2.
		\]
		Thus one may take $c_\alpha=\sqrt{\delta/C_n}$.
	\end{proof}
	
	\section{Marked period exclusion}
	\label{sec:collapsing-marked-period}

	The Green function argument above is independent of dimension: on any compact
	Calabi--Yau manifold, collapse to a point along a single sequence $\alpha+\kappa_i$,
	with $\kappa_i$ K\"ahler and $\kappa_i\to0$, forces rigidity.
	
	Starting in this section, we restrict to K3 surfaces and turn to the converse
	implication in Theorem~\ref{thm:equivalence}.
	We first combine the regular fibration theorem with the small cycle analysis
	of Ouyang--Tian and the fixed K3 marking to exclude limits of dimensions two
	and three.  We then construct test functions in the remaining one-dimensional
	case.
	
	Let $(X_i,g_i)$ be hyperk\"ahler K3 surfaces such that
	\[
	0<d_-\leq\diam(X_i,g_i)\leq d_+<\infty,
	\qquad
	\Vol(X_i,g_i)\longrightarrow0,
	\qquad
	(X_i,g_i)\xrightarrow{\mathrm{GH}}(Z,d_Z).
	\]
	Fix Gromov--Hausdorff $\epsilon_i$-approximations $\psi_i:X_i\to Z$, with
	$\epsilon_i\to0$.
	
	\subsection{Curvature regularity and local fibrations}
	
	Cheeger--Tian $\varepsilon$-regularity
	\cite[Theorem 0.8]{CheegerTian} gives universal constants
	$\varepsilon,C>0$ such that, on every Ricci-flat manifold of real dimension
	four with compact $\overline{B_r(p)}$,
	\[
	\int_{B_r(p)}|\operatorname{Rm}|^2\,dV\leq\varepsilon
	\quad\Longrightarrow\quad
	\sup_{B_{r/2}(p)}|\operatorname{Rm}|
	\leq Cr^{-2}.
	\]
	No lower volume bound is required.  Let $S\subset Z$ be the set of limits of
	sequences $q_i\in X_i$ for which
	$|\operatorname{Rm}_{g_i}(q_i)|\to\infty$.  Since
	Chern--Gauss--Bonnet gives total curvature energy $192\pi^2$, each point of
	$S$ accounts for at least $\varepsilon$ of this energy.  Hence
	$\#S\leq192\pi^2/\varepsilon$.  Moreover, on $\psi_i^{-1}(K)$ for every
	$K\Subset Z\setminus S$, the curvature and all its covariant derivatives are
	uniformly bounded, compare \cite[Corollary 3.22]{SunZhang}.
	
	The following proposition is due to Sun--Zhang
	\cite[Theorem 3.25 and Appendix A]{SunZhang}.  The higher derivative
	estimates are supplied by their Appendix~A.  Compatibility with the fixed
	Gromov--Hausdorff approximations is retained from the
	Cheeger--Fukaya--Gromov construction underlying the theorem and will be used
	in Section~\ref{sec:positive-limits}.
	
	\begin{proposition}
		\label{prop:regular-fibration}
		Let $Q\Subset Z\setminus S$ be a connected compact domain with smooth
		boundary in a smooth stratum of dimension $d$.
		For all large $i$, there are compact domains $Q_i\subset X_i$ with
		$\partial Q_i=F_i^{-1}(\partial Q)$ and smooth fiber bundle maps
		\[
		F_i:Q_i\longrightarrow Q
		\]
		such that $F_i$ is a Gromov--Hausdorff $\tau_i$ approximation,
		$\tau_i\to0$, and
		\begin{equation}
			|\nabla F_i|_{g_i}
			+|\nabla^2F_i|_{g_i}
			\leq C_Q.
			\label{eq:fibration-derivatives}
		\end{equation}
		The maps may be chosen compatibly with the fixed Gromov--Hausdorff
		approximations in the quantitative sense that
		\begin{equation}
			\sup_{x\in Q_i}d_Z\bigl(F_i(x),\psi_i(x)\bigr)\longrightarrow0.
			\label{eq:fibration-GH-compatibility}
		\end{equation}
		Moreover, for every $K\Subset\operatorname{int}Q$,
		$\psi_i^{-1}(K)\subset Q_i$ for all large $i$.
		Every fiber is a compact infranilmanifold of dimension $4-d$.  In
		particular, when $d=1$ the fibers have dimension three and are of flat or
		Heisenberg type, up to a finite affine quotient.
	\end{proposition}
	
	\begin{proof}
		Choose a compact domain $\widetilde Q$ with
		$Q\Subset\operatorname{int}\widetilde Q\Subset Z\setminus S$.
		The curvature bounds above allow us to apply Sun--Zhang's localized theorem
		over $\widetilde Q$.  Their construction gives the infranil fiber bundle
		and control as an almost Riemannian submersion
		\cite[Theorem 3.25]{SunZhang}, and their Appendix~A gives the stated derivative
		bounds.  The smoothed projection remains uniformly close to the original
		Gromov--Hausdorff approximation, as in the underlying
		Cheeger--Fukaya--Gromov construction
		\cite[(2.6.7)]{CGF}.  Restricting its full preimage to $Q$
		gives $Q_i$, the boundary identity, and
		\eqref{eq:fibration-GH-compatibility}.  If
		$K\Subset\operatorname{int}Q$, its positive distance from $\partial Q$,
		together with the buffer $\widetilde Q$, gives
		$\psi_i^{-1}(K)\subset Q_i$ for all large $i$.  No
		invariant hyperk\"ahler replacement of $g_i$ is used here.
	\end{proof}
	
	\subsection{Small primitive cycles in dimensions two and three}
	
	The following proposition is a consequence of Ouyang--Tian's analysis of
	limits of dimensions two and three
	\cite[Theorems 4.10 and 5.4]{OuyangTian}.  It provides the small cycles used
	in the marked period argument.
	
	\begin{proposition}
		\label{prop:small-period-cycles}
		Let $(X_i,g_i)$ be hyperk\"ahler K3 surfaces of unit diameter with
		hyperk\"ahler triples
		$\boldsymbol\omega_i=(\omega_{1,i},\omega_{2,i},\omega_{3,i})$ and volumes
		$V_i\to0$.  If their Gromov--Hausdorff limit has dimension two
		or three, then, after passing to a subsequence, there are embedded
		oriented tori $T_i\subset X_i$ with the following properties:
		\begin{enumerate}[label=\textup{(\roman*)}]
			\item $[T_i]\in H_2(X_i,\Z)$ is primitive and
			$T_i\cdot T_i=0$;
			\item for $a=1,2,3$,
			\begin{equation}
				\left|\int_{T_i}\omega_{a,i}\right|
				\leq CV_i.
				\label{eq:small-cycle-periods}
			\end{equation}
		\end{enumerate}
	\end{proposition}
	
	\begin{proof}
		Suppose first that the limit has dimension three.  In the proof of
		\cite[Theorem 5.4]{OuyangTian}, the inverse image of any one of three
		independent circles in the regular part of the limiting flat orbifold is
		an embedded torus $T_i$.  Its normal bundle is trivial, so
		$T_i\cdot T_i=0$, and its class is primitive by
		\cite[Lemma 5.6]{OuyangTian}.  If $\varepsilon_i$ denotes the period of
		the collapsing circle, the period computation in that proof gives
		\[
		\int_{T_i}\omega_{a,i}
		=\varepsilon_i\int_{p_0}^{p_0+v}dx_{a,i},
		\qquad
		\int_{X_i}\omega_{a,i}^2\asymp\varepsilon_i.
		\]
		The line integrals are bounded, while
		$\int_{X_i}\omega_{a,i}^2=2V_i$.  Thus
		$\varepsilon_i\asymp V_i$, which proves
		\eqref{eq:small-cycle-periods}.
		
		Suppose next that the limit has dimension two.  An $SO(3)$ hyperk\"ahler
		rotation preserves the Euclidean norm of the period vector, so it is enough
		to prove the estimate for the rotated triple.  After such a
		rotation, \cite[Theorem 4.10]{OuyangTian} gives an elliptic fibration
		\[
		\pi_i:X_i\longrightarrow\mathbb P_i^1.
		\]
		A smooth fiber $T_i$ has primitive class by
		\cite[Proposition 2.5]{OuyangTian} and self-intersection zero.  With the
		rotation chosen so that
		$\Omega_i=\omega_{1,i}+i\omega_{2,i}$ is holomorphic, put
		\[
		\mu_i=\int_{T_i}\omega_{3,i}>0.
		\]
		The restriction of $\Omega_i$ to $T_i$ vanishes, so the other
		two periods are zero.  Let $h_i'$ be the generalized Kähler--Einstein
		metric on the base associated with
		$\Omega_i/\sqrt{\mu_i}$.  Fiber integration gives
		\[
		\operatorname{Area}(\mathbb P_i^1,h_i')
		=\frac{4V_i}{\mu_i}.
		\]
		Choose $U_0\Subset U$ in the regular region with positive limit area, and
		let $f_i:U\to\mathbb P_i^1$ be the local identifications of
		\cite[(4.11)]{OuyangTian}.  By \cite[Lemma 4.13]{OuyangTian},
		$f_i^*h_i'$ converges in $C^{0,\alpha}$ on $U$ to the limit metric of
		dimension two.  Thus $\operatorname{Area}(f_i(U_0),h_i')\geq c>0$.  Hence
		$\mu_i\leq4c^{-1}V_i$, proving
		\eqref{eq:small-cycle-periods} also in this case.
	\end{proof}
	
	With these geometric inputs in place, we turn to the marked period plane and
	the arithmetic obstruction.
	
	\subsection{The marked period plane and its lattice minimum}
	
	We first fix the notation and normalizations.  The marked second cohomology
	lattice is
	\[
	\Lambda=H^2(X,\Z),\qquad
	q(v,w)=\int_X v\wedge w.
	\]
	It is the even unimodular K3 lattice of signature $(3,19)$.  We call
	$e\in\Lambda$ \emph{primitive} if $e=ke'$ with $e'\in\Lambda$ and
	$k\in\Z$ implies $k=\pm1$, and \emph{isotropic} if $q(e,e)=0$.
	
	Fix a nonzero holomorphic two-form $\Omega$ and put
	\[
	H=\operatorname{span}_{\R}\{\operatorname{Re}\Omega,
	\operatorname{Im}\Omega\}.
	\]
	The restriction of $q$ to $H$ is positive definite.  Real $(1,1)$-classes
	are precisely the vectors orthogonal to $H$, so
	\[
	N:=H^\perp\subset\Lambda_\R
	\]
	has signature $(1,19)$ and contains every $\beta_i$ and $\alpha$.
	
	Let $\beta_i\to\alpha$ be Kähler classes and set
	\[
	q_i=q(\beta_i,\beta_i)>0,\qquad
	u_i=\frac{\beta_i}{\sqrt{q_i}},\qquad
	P_i=H\oplus\R u_i.
	\]
	Thus $q(u_i,u_i)=1$, and $P_i$ is an oriented positive three-plane.  This
	plane has the following metric interpretation.  Complete the Ricci-flat
	Kähler form in $\beta_i$
	to a hyperk\"ahler triple $\boldsymbol\omega_i=(\omega_{1,i},
	\omega_{2,i},\omega_{3,i})$, with $\omega_{3,i}$ in the class $\beta_i$.
	The first two cohomology classes span $H$, and therefore
	\[
	P_i=\operatorname{span}_{\R}\{[\omega_{1,i}],[\omega_{2,i}],
	[\omega_{3,i}]\}.
	\]
	Multiplying the metric by a positive constant multiplies all three
	Kähler forms by that constant and does not change $P_i$.  Hyperkähler
	rotation only changes the orthonormal basis of this same plane.
	
	For an arbitrary positive three-plane $P\subset\Lambda_\R$, write the
	$q$-orthogonal decomposition as
	\[
	v=v_P+v_{P^\perp}.
	\]
	Since $q$ is positive definite on $P$ and negative definite on $P^\perp$,
	the formula
	\begin{equation}
		\|v\|_P^2=q(v_P,v_P)-q(v_{P^\perp},v_{P^\perp})
		\label{eq:Hodge-norm}
	\end{equation}
	defines a positive definite norm, called the Hodge norm associated with
	$P$.  If $q(e,e)=0$, then
	\[
	q(e_P,e_P)+q(e_{P^\perp},e_{P^\perp})=0,
	\]
	and hence
	\begin{equation}
		\|e\|_P^2=2q(e_P,e_P)
		=2\|\operatorname{pr}_Pe\|_q^2.
		\label{eq:isotropic-Hodge-projection}
	\end{equation}
	Here $\|\cdot\|_q$ denotes the Euclidean norm defined by the positive form
	$q|_P$.
	
	Define the shortest primitive isotropic projection by
	\begin{equation}
		m(P)=\min_{\substack{e\in\Lambda\ \mathrm{primitive}\\q(e,e)=0}}
		\|\operatorname{pr}_Pe\|_q.
		\label{eq:lattice-minimum}
	\end{equation}
	This set is nonempty because the K3 lattice contains a hyperbolic plane
	$U$, and a standard generator of $U$ is primitive and isotropic.  The
	minimum is positive and is attained.  Indeed, $\|\cdot\|_P$ is a
	positive definite norm on the discrete lattice $\Lambda$, so every
	$\|\cdot\|_P$-bounded set contains only finitely many lattice points,
	now use \eqref{eq:isotropic-Hodge-projection}.
	
	Geometric periods are related to orthogonal projection as follows.  Let $g$
	be a hyperk\"ahler metric of volume $V$ and
	let $\boldsymbol\omega=(\omega_1,\omega_2,\omega_3)$ be its hyperk\"ahler
	triple.  Our convention gives
	\[
	q([\omega_a],[\omega_b])
	=\int_X\omega_a\wedge\omega_b=2V\delta_{ab}.
	\]
	Thus $p_a=[\omega_a]/\sqrt{2V}$ is a $q$-orthonormal basis of its period
	plane $P$.  If an oriented two-cycle $T$ has Poincaré dual $e$, then
	\begin{equation}
		\|\operatorname{pr}_Pe\|_q^2
		=\sum_{a=1}^3q(e,p_a)^2
		=\frac{1}{2V}\sum_{a=1}^3
		\left(\int_T\omega_a\right)^2.
		\label{eq:period-projection}
	\end{equation}
	Formula \eqref{eq:period-projection} is the normalization used below.
	
	\subsection{The arithmetic lower bound}
	
	The fixed marking enters through the following lemma.
	
	\begin{lemma}
		\label{lem:arithmetic}
		If $\R\alpha\cap H^2(X,\Q)=\{0\}$, then
		\[
		\frac{m(P_i)}{\sqrt{q_i}}\longrightarrow\infty.
		\]
	\end{lemma}
	
	\begin{proof}
		Choose a $q$-unit vector $u_*\in N$ in the same component of the positive
		cone as the Kähler cone and put $P_*=H\oplus\R u_*$.  We write
		$\|\cdot\|_*=\|\cdot\|_{P_*}$.  We first compare this norm with the
		varying norms $\|\cdot\|_{P_i}$.
		
		Since $u_i$ and $u_*$ are unit positive vectors in the Lorentzian space
		$N$, there are $r_i\ge0$ and a vector $w_i\in u_*^\perp\cap N$ such that
		\[
		q(w_i,w_i)=-1,
		\qquad
		u_i=(\cosh r_i)u_*+(\sinh r_i)w_i.
		\]
		When $r_i=0$, the choice of $w_i$ is immaterial.  Put
		$L_i=\operatorname{span}_{\R}\{u_*,w_i\}$, and define $B_i\in O(q)$ to be
		the identity on $L_i^{\perp_q}$ and to act on $L_i$ by
		\[
		B_iu_*=u_i=(\cosh r_i)u_*+(\sinh r_i)w_i,
		\qquad
		B_iw_i=(\sinh r_i)u_*+(\cosh r_i)w_i.
		\]
		Since $H\subset L_i^{\perp_q}$, the map $B_i$ fixes $H$ and carries $P_*$
		to $P_i$.  A $q$-isometry carrying one positive plane to another also carries
		the corresponding $q$-orthogonal decompositions in
		\eqref{eq:Hodge-norm} to one another.  Consequently,
		\[
		\|B_ix\|_{P_i}=\|x\|_*,
		\qquad\text{or equivalently}\qquad
		\|v\|_{P_i}=\|B_i^{-1}v\|_*.
		\]
		For the null vectors $\ell_\pm=u_*\pm w_i$, one has
		$B_i\ell_\pm=e^{\pm r_i}\ell_\pm$.  In the positive definite inner
		product associated with $\|\cdot\|_*$, the vectors $u_*$ and $w_i$ are
		orthonormal: $u_*\in P_*$, while $w_i\in P_*^\perp$ and the sign of $q$ is
		reversed on $P_*^\perp$.  Thus $\ell_+/\sqrt{2}$ and
		$\ell_-/\sqrt{2}$ are orthonormal, and their orthogonal complement is
		$L_i^{\perp_q}$, where $B_i$ is the identity.  Hence
		$\|B_i\|_{*\to *}=e^{r_i}$.  Applying this to $B_i^{-1}v$ gives
		\begin{equation}
			\|v\|_*
			=\|B_i(B_i^{-1}v)\|_*
			\le e^{r_i}\|B_i^{-1}v\|_*
			=e^{r_i}\|v\|_{P_i}
			\qquad(v\in\Lambda_\R).
			\label{eq:boost-comparison}
		\end{equation}
		Moreover,
		\[
		\cosh r_i=q(u_i,u_*)
		=\frac{q(\beta_i,u_*)}{\sqrt{q_i}}.
		\]
		Since $q(\beta_i,u_*)\to q(\alpha,u_*)>0$, the sequence
		$q(\beta_i,u_*)$ is bounded.  Since $e^{r_i}\le2\cosh r_i$,
		\eqref{eq:boost-comparison}
		implies
		\begin{equation}
			\|v\|_*\le \frac{C}{\sqrt{q_i}}\|v\|_{P_i}.
			\label{eq:norm-comparison}
		\end{equation}
		
		Suppose the conclusion of the lemma were false.  After passing to a
		subsequence, there would be a constant $A$ and primitive isotropic
		vectors $e_i\in\Lambda$ realizing the minimum such that
		\[
		\|\operatorname{pr}_{P_i}e_i\|_q=m(P_i)
		\le A\sqrt{q_i}.
		\]
		Equations \eqref{eq:isotropic-Hodge-projection} and
		\eqref{eq:norm-comparison} give
		\[
		\|e_i\|_*
		\le \frac{C}{\sqrt{q_i}}m(P_i)
		\le CA.
		\]
		Only finitely many vectors of $\Lambda$ lie in this fixed ball.  Passing
		to another subsequence, we may therefore assume $e_i=e\ne0$.
		
		Because $q_i\to q(\alpha,\alpha)=0$, the displayed upper bound on
		$m(P_i)$ tends to zero.  Since $P_i=H\oplus\R u_i$ is a
		$q$-orthogonal sum, the $H$-component of
		$\operatorname{pr}_{P_i}e$ is the fixed vector
		$\operatorname{pr}_He$, and
		\[
		\|\operatorname{pr}_He\|_q
		\leq\|\operatorname{pr}_{P_i}e\|_q
		=m(P_i)\longrightarrow0.
		\]
		Thus $\operatorname{pr}_He=0$, or equivalently $e\in H^\perp=N$.  Its
		remaining component along $u_i$ also tends to zero:
		\[
		|q(e,u_i)|\le\|\operatorname{pr}_{P_i}e\|_q
		\longrightarrow0.
		\]
		Consequently
		\[
		q(e,\alpha)=\lim_{i\to\infty}q(e,\beta_i)
		=\lim_{i\to\infty}\sqrt{q_i}\,q(e,u_i)=0.
		\]
		
		It remains to use a basic fact about Lorentzian spaces.  In a vector space
		of signature $(1,n)$, two nonzero isotropic vectors which are orthogonal
		are proportional.  For completeness, choose coordinates in which
		$q(x,x)=x_0^2-|x'|^2$.  If $x$ and $y$ are isotropic and $q(x,y)=0$, then
		$|x_0|=|x'|$, $|y_0|=|y'|$, and equality holds in the Euclidean
		Cauchy--Schwarz inequality for $x'$ and $y'$, hence $x'$ and $y'$ are
		proportional, with the same proportionality for the zeroth coordinates.
		Applying this to $e,\alpha\in N$ shows that $e=c\alpha$.  Since
		$0\ne e\in H^2(X,\Z)$, the ray $\R\alpha$ is rational, contradicting the
		hypothesis.
	\end{proof}
	
	\subsection{The geometric upper bound and the exclusion of dimensions two and three}
	
	We first translate Proposition~\ref{prop:small-period-cycles} into a
	period plane estimate with explicit diameter dependence.  This form applies,
	in particular, to collapsing sequences whose diameters remain bounded above
	and below. In the argument below, only the lower bound is used.
	
	\begin{lemma}
		\label{lem:OT-scale}
		Let $(X_i,g_i)$ be hyperk\"ahler K3 surfaces, and put
		\[
		D_i=\diam(X_i,g_i),
		\qquad
		V_i=\Vol(X_i,g_i).
		\]
		Suppose that $D_i^{-4}V_i\to0$.  Fix markings
		$H^2(X_i,\Z)\cong\Lambda$, and denote the period planes by $P_i$.  If
		the metrics $D_i^{-2}g_i$, normalized to unit diameter, have a
		Gromov--Hausdorff limit of
		dimension two or three, then, after passing to a subsequence, there are
		primitive isotropic vectors $e_i\in\Lambda$ such that
		\[
		\|\operatorname{pr}_{P_i}e_i\|_q
		\leq C D_i^{-2}\sqrt{V_i}.
		\]
		In particular, $m(P_i)\leq C D_i^{-2}\sqrt{V_i}$.
	\end{lemma}
	
	\begin{proof}
		Put
		\[
		\widehat g_i=D_i^{-2}g_i,
		\qquad
		\widehat{\boldsymbol\omega}_i=D_i^{-2}\boldsymbol\omega_i,
		\qquad
		\widehat V_i=D_i^{-4}V_i,
		\]
		where $\boldsymbol\omega_i$ is a hyperk\"ahler triple for $g_i$.  Multiplying
		all three components by the same positive constant leaves the period plane
		$P_i$ unchanged.  Let $T_i$ be given by Proposition
		\ref{prop:small-period-cycles} for the rescaled sequence, and let
		$e_i$ be its Poincaré dual under the chosen marking.  Then $e_i$ is
		primitive and isotropic.  Equations \eqref{eq:period-projection} and
		\eqref{eq:small-cycle-periods} give
		\[
		\|\operatorname{pr}_{P_i}e_i\|_q^2
		=\frac{1}{2\widehat V_i}
		\sum_{a=1}^3
		\left(\int_{T_i}\widehat\omega_{a,i}\right)^2
		\leq C\widehat V_i
		=C D_i^{-4}V_i.
		\]
		Taking square roots proves the displayed estimate, and the assertion about
		$m(P_i)$ follows from its definition.
	\end{proof}
	
	\begin{proof}[Proof of Theorem~\ref{thm:dimension-exclusion}]
		Suppose, to the contrary, that a subsequence of the metrics in the statement,
		normalized to unit diameter, converges to a limit of dimension two or three.
		Put
		$D_i=\diam(X,g_i)$ and $V_i=\Vol(X,g_i)=q_i/2$.  Since $q_i\to0$ and
		$D_i$ is bounded below,
		\[
		D_i^{-4}V_i=D_i^{-4}\frac{q_i}{2}\longrightarrow0.
		\]
		Lemma~\ref{lem:OT-scale} therefore gives
		\[
		m(P_i)\le CD_i^{-2}\sqrt{\frac{q_i}{2}}
		\le C'\sqrt{q_i},
		\]
		contradicting Lemma~\ref{lem:arithmetic}.
	\end{proof}
	
	\section{Limits of positive dimension produce nonrigidity}
	\label{sec:positive-limits}
	
	In this section we prove \textup{(i)}$\Rightarrow$\textup{(ii)} in
	Theorem~\ref{thm:equivalence}.  It is enough to show that
	$\alpha$ is nonrigid whenever there is a sequence of Kähler classes
	$\beta_i\to\alpha$ with $\liminf_{i\to\infty}D(\beta_i)>0$.  There are two
	cases. A rational nef isotropic ray is nonrigid by an elementary argument from
	algebraic geometry.  For an irrational ray, Theorem
	\ref{thm:dimension-exclusion} removes the case with limits of dimensions two and three. Thus it remains only to consider a one-dimensional
	limit.
	
	\begin{lemma}
		\label{lem:rational-ray-nonrigid}
		Let $0\ne\alpha\in\overline{\Kah_X}$ satisfy $\alpha^2=0$.  If
		\[
		\R\alpha\cap H^2(X,\Q)\ne\{0\},
		\]
		then $\alpha$ is nonrigid.
	\end{lemma}
	
	\begin{proof}
		Choose $\lambda>0$ such that
		$e:=\lambda\alpha\in H^2(X,\Z)$ is primitive.  Since $e$ is an integral
		$(1,1)$-class, there is a line bundle $L$ with $c_1(L)=e$.  It is
		nontrivial and nef, and Riemann--Roch on a K3 surface gives
		\[
		\chi(X,L)=2+\frac12e^2=2.
		\]
		If $L^{-1}$ had a nonzero section, its divisor would be an effective
		representative of $-e$, which has negative intersection with every
		Kähler class.  This is impossible.  Hence, by Serre duality,
		$h^2(X,L)=h^0(X,L^{-1})=0$, and Riemann--Roch yields
		$h^0(X,L)=2+h^1(X,L)\geq2$.
		
		Choose nonproportional sections $s_0,s_1\in H^0(X,L)$.  Their zero
		divisors $D_0,D_1$ are distinct: otherwise $s_0/s_1$ would be a meromorphic
		function with neither zeros nor poles, hence holomorphic and constant.  The
		integration currents $[D_0]$ and $[D_1]$ are therefore distinct closed
		positive currents in the class $e$.  Scaling them by $\lambda^{-1}$
		gives two distinct positive currents in $\alpha$.
	\end{proof}
	
	\subsection{Test functions in a one-dimensional collapse}
	
	We now isolate the consequence of Sun--Zhang's localized regular fibration
	theorem that is used in the proof of Proposition
	\ref{prop:positive-limit}.  
	
	\begin{lemma}
		\label{lem:interval-test-functions}
		Let $(X_i,g_i)$ be hyperk\"ahler K3 surfaces satisfying
		\[
		0<d_-\leq\diam(X_i,g_i)\leq d_+<\infty,
		\qquad
		\Vol(X_i,g_i)\longrightarrow0,
		\]
		and suppose that
		\[
		(X_i,g_i)\xrightarrow{\mathrm{GH}}(Z,d_Z),
		\]
		where $Z$ is a compact interval.  Put
		\[
		\mu_i=\frac{dV_{g_i}}{\Vol(X_i,g_i)}.
		\]
		Fix Gromov--Hausdorff $\epsilon_i$-approximations $\psi_i:X_i\to Z$
		realizing this convergence, and let $S\subset Z$ be the curvature
		concentration set for $g_i$.  Choose compact intervals with nonempty interior
		\[
		Q\Subset\operatorname{int}Q'
		\Subset\operatorname{int}Z\setminus S
		\]
		and a function
		\[
		h\in C_c^\infty(\operatorname{int}Q),
		\qquad 0\leq h\leq1.
		\]
		Assume that, for $a=0,1$, there is a nonempty open interval
		$U_a\Subset\operatorname{int}Q$ on which $h\equiv a$.
		Extend $h$ by zero to a smooth function on $Q'$.  Let
		$F_i:Q_i'\to Q'$ be the regular fibration maps given by Proposition
		\ref{prop:regular-fibration}.  The
		functions $u_i=h\circ F_i$ extend by zero to smooth functions on $X_i$ and
		satisfy
		\begin{equation}
			0\leq u_i\leq1,
			\qquad
			|\nabla^2u_i|_{g_i}\leq C_h.
			\label{eq:test-function-Hessian}
		\end{equation}
		Moreover, there is $\delta>0$ such that, for all sufficiently large $i$,
		\begin{equation}
			\mu_i(\{u_i=0\})\geq\delta,
			\qquad
			\mu_i(\{u_i=1\})\geq\delta.
			\label{eq:plateau-measures}
		\end{equation}
	\end{lemma}
	
	\begin{proof}
		Proposition~\ref{prop:regular-fibration}, applied to $Q'$, gives the
		derivative bounds, compatibility with $\psi_i$, and interior coverage used
		below.
		
		Since $h$ vanishes on a neighborhood of $Q'\setminus Q$ and
		$\partial Q_i'=F_i^{-1}(\partial Q')$, extension by zero is smooth.
		The chain rule and \eqref{eq:fibration-derivatives} give
		\eqref{eq:test-function-Hessian}.
		
		Choose $z_a\in U_a$, $a=0,1$, and $\rho>0$ such that
		$\overline{B_{4\rho}(z_a)}\subset U_a$ and $\rho\leq2d_+$.  By the
		approximate surjectivity of $\psi_i$, choose $x_{a,i}\in X_i$ with
		$d_Z(\psi_i(x_{a,i}),z_a)<\epsilon_i$.  For all large $i$, the distortion
		bound for $\psi_i$ sends $B_\rho(x_{a,i})$ into $B_{2\rho}(z_a)$.
		Applying the interior coverage assertion to
		$\overline{B_{2\rho}(z_a)}$ places this ball in $Q_i'$.  Since
		$\sup_{Q_i'}d_Z(F_i,\psi_i)<\rho$ for all large $i$, its $F_i$-image lies
		in $B_{3\rho}(z_a)\subset U_a$, and hence
		\[
		B_\rho(x_{a,i})\subset\{u_i=a\}.
		\]
		Bishop--Gromov comparison, $\Ric(g_i)=0$, and
		$\diam(X_i,g_i)\leq d_+$ now give
		\[
		\mu_i(\{u_i=a\})
		\geq\mu_i(B_\rho(x_{a,i}))
		\geq\left(\frac{\rho}{2d_+}\right)^4,
		\qquad a=0,1.
		\]
		This proves \eqref{eq:plateau-measures}.
	\end{proof}
	
	\begin{proposition}
		\label{prop:positive-limit}
		Let $0\ne\alpha\in\overline{\Kah_X}$ satisfy $\alpha^2=0$, and let
		$\beta_i\to\alpha$ be Kähler classes.  If
		\[
		\liminf_{i\to\infty}D(\beta_i)>0,
		\]
		then $\alpha$ is nonrigid.
	\end{proposition}
	
	\begin{proof}
		If the ray of $\alpha$ is rational, the conclusion follows from Lemma
		\ref{lem:rational-ray-nonrigid}.  We may therefore assume
		$\R\alpha\cap H^2(X,\Q)=\{0\}$.  Set $\omega_i=\CY(\beta_i)$ and
		$g_i=g_{\omega_i}$.
		After passing to a subsequence, the hypothesis and the uniform diameter
		theorem \cite[Theorem 4.1]{Tosatti} give constants $d_0,D_0>0$ such that
		$d_0\leq D(\beta_i)\leq D_0$.
		By Gromov compactness, after passing to a further subsequence, we obtain
		\[
		(X,g_i)\xrightarrow{\mathrm{GH}}(Z,d_Z).
		\]
		Fix Gromov--Hausdorff $\epsilon_i$-approximations $\psi_i:X\to Z$, and set
		$\mu_i=dV_{g_i}/\Vol(X,g_i)$.  Continuity of diameter gives
		$d_0\leq\diam Z\leq D_0$, while
		$\Vol(X,g_i)=\beta_i^2/2\to0$.
		Sun--Zhang's classification gives $\dim Z\in\{1,2,3\}$ and identifies a
		one-dimensional limit with a compact interval
		\cite[Theorem 1.1]{SunZhang}.  Theorem
		\ref{thm:dimension-exclusion} excludes dimensions two and three.  Hence $Z$
		is a compact interval.  Let $S\subset Z$ be the curvature concentration set
		associated with this subsequence and the maps $\psi_i$.
		
		Choose compact intervals
		$Q\Subset\operatorname{int}Q'\Subset\operatorname{int}Z\setminus S$ and a
		function
		\[
		h\in C_c^\infty(\operatorname{int}Q),
		\qquad
		0\leq h\leq1.
		\]
		Choose $h$ so that each of $\{h=0\}$ and $\{h=1\}$ contains a nonempty
		open interval compactly contained in $\operatorname{int}Q$.
		Let $u_i$ be the pullback functions supplied by Lemma
		\ref{lem:interval-test-functions}, extended by zero.
		
		\paragraph{Uniform complex Hessian control.}
		Equation \eqref{eq:test-function-Hessian} gives
		\[
		|\nabla^2_{g_i}u_i|_{g_i}\leq C_h.
		\]
		The complex structure $J$ is fixed and
		$\nabla^{g_i}J=0$.  Hence, after enlarging the constant, there is $C_1$
		independent of $i$ such that
		\begin{equation}
			-C_1\omega_i\leq\ddc u_i\leq C_1\omega_i.
			\label{eq:ddcu-bound}
		\end{equation}
		Choose $0<\varepsilon<C_1^{-1}$.  Then
		\begin{equation}
			\omega_i^\pm
			:=\omega_i\pm\varepsilon\ddc u_i
			\label{eq:perturbed-forms}
		\end{equation}
		are Kähler forms and $[\omega_i^\pm]=\beta_i$.
		
		\paragraph{$L^1$ compactness on the fixed complex surface.}
		Choose smooth closed representatives $\theta_i\in\beta_i$ such that
		\[
		\theta_i\longrightarrow\theta_\alpha
		\qquad\text{in }C^\infty(X),
		\]
		and write
		\[
		\omega_i=\theta_i+\ddc\varphi_i,
		\qquad
		\sup_X\varphi_i=0.
		\]
		For a fixed background Kähler form $\omega_0$, the functions
		$\varphi_i$ satisfy
		\[
		\ddc\varphi_i\geq-\theta_i\geq-C\omega_0.
		\]
		Compactness of normalized quasi-plurisubharmonic functions therefore
		gives, after passing to a subsequence,
		\begin{equation}
			\varphi_i\longrightarrow\varphi
			\qquad\text{in }L^1(X).
			\label{eq:phi-convergence}
		\end{equation}
		Set
		\[
		v_i=u_i+C_1\varphi_i.
		\]
		By \eqref{eq:ddcu-bound},
		\[
		\ddc v_i
		\geq-C_1\omega_i+C_1(\omega_i-\theta_i)
		=-C_1\theta_i
		\geq-C_2\omega_0.
		\]
		Since $0\leq u_i\leq1$ and $\sup_X\varphi_i=0$, one also has
		$0\leq\sup_Xv_i\leq1$.  A further application of quasi-psh compactness
		gives $v_i\to v$ in $L^1(X)$.  Consequently,
		\begin{equation}
			u_i\longrightarrow u:=v-C_1\varphi
			\qquad\text{in }L^1(X).
			\label{eq:u-convergence}
		\end{equation}
		
		\paragraph{The limiting function is nonconstant.}
		By Subsection \ref{subsec:fixed-measure}, the normalized volume measures
		of $g_i$ all equal the same smooth probability measure $\mu$.  This
		measure has a smooth bounded density with
		respect to the fixed background volume $\omega_0^2$, so
		\eqref{eq:u-convergence} also holds in $L^1(X,\mu)$.  By
		\eqref{eq:plateau-measures}, there is $\delta>0$ such that, for all large $i$,
		\[
		\mu(\{u_i=0\})\geq\delta,
		\qquad
		\mu(\{u_i=1\})\geq\delta.
		\]
		If $u$ were constant, say $u\equiv c$, then $0\leq c\leq1$ and
		\[
		\|u_i-c\|_{L^1(X,\mu)}
		\geq c\,\mu(\{u_i=0\})
		+(1-c)\,\mu(\{u_i=1\})
		\geq\delta,
		\]
		contrary to \eqref{eq:u-convergence}.  Thus
		$u$ is nonconstant.

		\paragraph{Two distinct limiting positive currents.}
		The masses of $\omega_i$ with respect to a fixed background Kähler form
		are uniformly bounded because $\beta_i\to\alpha$.  After passing to a
		subsequence,
		\[
		\omega_i\rightharpoonup T
		\]
		for a closed positive $(1,1)$-current $T$ with $[T]=\alpha$.  Equation
		\eqref{eq:u-convergence} gives
		$\ddc u_i\rightharpoonup\ddc u$ in the sense of distributions.
		Passing to the limit in \eqref{eq:perturbed-forms} yields
		\[
		T^\pm:=T\pm\varepsilon\ddc u\geq0,
		\qquad
		[T^\pm]=\alpha.
		\]
		If $T^+=T^-$, then $\ddc u=0$. By compactness of $X$, $u$ is constant, contradicting the fact that $u$ is nonconstant. Therefore
		$T^+\neq T^-$, and $\alpha$ is nonrigid.
	\end{proof}
	
	\subsection{Proof and consequences of the metric characterization}
	
	\begin{proof}[Proof of Theorem~\ref{thm:equivalence}]
		The implication \textup{(ii)}$\Rightarrow$\textup{(iii)} is immediate.
		Applying Theorem~\ref{thm:point-collapse-general-CY} to
		$\kappa_i=i^{-1}\kappa$ gives
		\textup{(iii)}$\Rightarrow$\textup{(i)}.
		
		It remains to prove
		\textup{(i)}$\Rightarrow$\textup{(ii)}.  Suppose that $\alpha$ is rigid
		but \textup{(ii)} fails.  Then there is a sequence $\beta_i\to\alpha$ for
		which $D(\beta_i)$ does not tend to zero.  After passing to a subsequence,
		there is a constant $d_0>0$ such that
		\[
		D(\beta_i)\geq d_0.
		\]
		Proposition~\ref{prop:positive-limit} implies that $\alpha$ is nonrigid,
		a contradiction.
	\end{proof}
	
	\begin{proof}[Proof of Corollary~\ref{cor:dichotomy}]
		The rigid case follows from Theorem~\ref{thm:equivalence}.  Suppose that
		$\alpha$ is nonrigid and set $g_t=g_{\CY(\alpha+t\kappa)}$.  Corollary
		\ref{cor:radial-green} and the general diameter theorem give
		\[
		0<c\leq\diam(X,g_t)\leq C
		\]
		for all small $t$.  Given any sequence $t_i\to0^+$, Gromov compactness
		gives, after passing to a subsequence,
		\[
		(X,g_{t_i})\xrightarrow{\mathrm{GH}}(Z,d_Z).
		\]
		Since
		$\Vol(X,g_{t_i})=t_i\alpha\cdot\kappa+t_i^2\kappa^2/2\to0$,
		Sun--Zhang's classification \cite[Theorem 1.1]{SunZhang} gives
		$\dim Z\in\{1,2,3\}$.  Theorem~\ref{thm:dimension-exclusion} excludes
		dimensions two and three.  Hence $Z$ is a compact interval, and the lower
		diameter bound makes it nondegenerate.
	\end{proof}
	
	\begin{remark}
		\label{rem:tangential}
		Theorem~\ref{thm:equivalence} should not be confused with the false
		statement
		\[
		\bigl(\exists\,\beta_i\to\alpha:D(\beta_i)\to0\bigr)
		\quad\Longrightarrow\quad
		\alpha\text{ is rigid}.
		\]
		Filip--Tosatti constructed a nonrigid rational elliptic class $\alpha$ and
		two sequences $\alpha_i,\beta_i\to\alpha$ satisfying
		\[
		\diam\bigl(X,\CY(\alpha_i)\bigr)\geq c>0,
		\qquad
		\diam\bigl(X,\CY(\beta_i)\bigr)\to0.
		\]
		This does not contradict Theorem
		\ref{thm:point-collapse-general-CY}: that theorem requires
		$\beta_i-\alpha=\kappa_i\in\Kah_X$.  The first sequence can be taken
		radial, $\alpha_i=\alpha+i^{-1}\omega$, whereas the second is produced by a
		parabolic automorphism and does not satisfy this condition,
		see \cite{FilipTosatti} and the discussion in
		\cite[Theorem 3.1 and (25)]{Tosatti}.
	\end{remark}
	
	\section{Examples with nonrigid classes and diameter bounded below}
	\label{sec:kummer}
	
	\subsection{Kummer construction}
	Fix $\theta\in\R\setminus\Q$.  Identifying $\C^2$ with $\R^4$ by the real
	coordinates $(x_1,x_2,y_1,y_2)$, let $A=\C^2/\Lambda$, where
	\[
	\Lambda=B\Z^4,
	\qquad
	B=
	\begin{pmatrix}
		1 & \theta & 0 & 0\\
		0 & 1 & 0 & 0\\
		0 & 0 & 1 & 0\\
		0 & 0 & 0 & 1
	\end{pmatrix}.
	\]
	The columns of $B$ correspond to
	\[
	\lambda_1=(1,0),\qquad
	\lambda_2=(\theta,1),\qquad
	\lambda_3=(i,0),\qquad
	\lambda_4=(0,i).
	\]
	The periods of $dx_1$ and $dy_1$ on these generators are
	$(1,\theta,0,0)$ and $(0,0,1,0)$, respectively.  Thus $y_1\bmod\Z$
	defines a map $A\to\R/\Z$.  Set
	\[
	\alpha_A=dx_1\wedge dy_1=\frac{i}{2}dz_1\wedge d\overline z_1.
	\]
	Its periods on the cycles $\lambda_i\wedge\lambda_j$, ordered as
	$(12,13,14,23,24,34)$, are
	\[
	(0,1,0,\theta,0,0).
	\]
	Thus $\R[\alpha_A]\cap H^2(A,\Q)=\{0\}$.
	
	Let $Y=A/\{\pm1\}$ and let $r:X\to Y$ be the minimal resolution.  Choose
	a smooth even function $h:\R/\Z\to[0,\infty)$ such that
	\[
	\int_0^1h(y)\,dy=1,
	\qquad h=0\text{ near }\tfrac12\Z.
	\]
	The form
	\[
	h(y_1)\alpha_A
	\]
	is closed, semipositive, invariant under $z\mapsto-z$, and zero near all
	fixed points.  It therefore descends and extends by zero to a smooth
	semipositive form $T_h$ on $X$.
	
	\begin{proof}[Proof of Theorem~\ref{thm:kummer}\textup{(a)}]
		If $h_1,h_2$ are two such functions, set $k=h_1-h_2$ and
		\[
		H(y)=\int_0^y k(s)\,ds.
		\]
		Evenness and equal averages imply that $H$ is periodic, odd, and zero near
		$\tfrac12\Z$.  The form $-H(y_1)dx_1$ descends smoothly to $X$, and
		\[
		T_{h_1}-T_{h_2}=d\bigl(-H(y_1)dx_1\bigr).
		\]
		Hence all $T_h$ represent one class $\alpha\in H^{1,1}(X,\R)$.  Since $T_h$
		is semipositive and has positive pairing with every Kähler class, $\alpha$ is
		nonzero and nef. Moreover, $T_h^2=0$, so $\alpha^2=0$.  If a nonzero multiple
		of $\alpha$ were rational, its pullback to the twofold torus cover away from
		the exceptional curves would make the same multiple of $[\alpha_A]$ rational.
		The displayed periods rule this out.  Thus the ray of $\alpha$ is irrational.
		Choosing two distinct admissible functions $h_1$ and $h_2$ gives
		$T_{h_1}\neq T_{h_2}$.  Hence $\alpha$ is nonrigid.
	\end{proof}
	
	\begin{proof}[Proof of Theorem~\ref{thm:kummer}\textup{(b)}]
		The diameter bound in Theorem~\ref{thm:kummer}(b) follows directly from
		Corollary~\ref{cor:radial-green}, since $\alpha$ is nonrigid.  We give a
		second proof that yields an explicit constant.  Choose a nonconstant smooth
		even periodic function $f$ with $f'=0$ near $\tfrac12\Z$.  Then
		$F=f(y_1)$ descends smoothly to $X$.  On the torus cover,
		\[
		\beta:=i\partial F\wedge\overline\partial F
		=\frac12 f'(y_1)^2\alpha_A.
		\]
		The preceding exactness calculation therefore gives
		\[
		[\beta]=c_f\alpha,
		\qquad
		c_f=\frac12\int_0^1f'(y)^2\,dy>0.
		\]
		Let $\kappa\in\Kah_X$ be arbitrary and put
		\[
		\omega_\kappa=\CY(\alpha+\kappa),\quad
		A_\kappa=\alpha\cdot\kappa,\quad B_\kappa=\kappa^2.
		\]
		Since $\alpha$ is nef, $\alpha+\kappa$ is Kähler.  Since $\alpha$ is
		nonzero, the Hodge index theorem gives
		$A_\kappa>0$. Moreover, $B_\kappa>0$ because $\kappa$ is Kähler.
		Then
		\[
		V_\kappa:=\Vol(X,\omega_\kappa)
		=A_\kappa+\frac12B_\kappa
		\]
		and the Dirichlet energy is cohomological:
		\[
		\int_X|dF|_{\omega_\kappa}^2\,dV_{\omega_\kappa}
		=2\int_X\beta\wedge\omega_\kappa=2c_f A_\kappa.
		\]
		The normalized Ricci-flat volume measure is independent of $\kappa$.  Its
		lift to the torus, normalized to have total mass one, is Haar measure on $A$.
		The pushforward of Haar measure by the nonconstant homomorphism
		$y_1\colon A\to\R/\Z$ is normalized Lebesgue measure.  Put
		\[
		\overline f=\int_0^1f(s)\,ds,
		\qquad
		F_0=F-\overline f,
		\qquad
		\sigma_f
		:=\sqrt{\int_0^1\bigl(f(y)-\overline f\bigr)^2\,dy}>0.
		\]
		The positivity of $\sigma_f$ follows from the choice of the nonconstant
		function $f$.  The preceding description of the measure gives
		\[
		\int_XF_0\,dV_{\omega_\kappa}=0,
		\qquad
		\int_XF_0^2\,dV_{\omega_\kappa}=V_\kappa\sigma_f^2>0,
		\qquad
		dF_0=dF.
		\]
		Let
		$\Delta_\kappa=d_{\omega_\kappa}^{*}d$ be the nonnegative Laplace--Beltrami
		operator on functions, and let $\lambda_1(\omega_\kappa)$ be its first positive
		eigenvalue.  For a nonzero smooth real function $v$, its Rayleigh
		quotient is
		\[
		\mathcal R_\kappa(v)
		:=\frac{\int_X|dv|_{\omega_\kappa}^2\,dV_{\omega_\kappa}}
		{\int_Xv^2\,dV_{\omega_\kappa}}.
		\]
		Because $X$ is connected, the zero eigenspace consists of the constant
		functions.  The variational characterization of the first positive
		eigenvalue is therefore
		\[
		\lambda_1(\omega_\kappa)
		=\inf\left\{\mathcal R_\kappa(v):
		0\ne v\in C^\infty(X,\R),\ 
		\int_Xv\,dV_{\omega_\kappa}=0\right\}.
		\]
		The function $F_0$ is an admissible test function.  Using $dF_0=dF$ and the
		Dirichlet energy identity above, we obtain
		\[
		\lambda_1(\omega_\kappa)
		\le\mathcal R_\kappa(F_0)
		=\frac{\int_X|dF|_{\omega_\kappa}^2\,dV_{\omega_\kappa}}
		{\int_XF_0^2\,dV_{\omega_\kappa}}
		=\frac{2c_f A_\kappa}{V_\kappa\sigma_f^2}.
		\]
		On the other hand, $(X,\omega_\kappa)$ is closed and
		$\Ric(\omega_\kappa)=0$.  The Li--Yau estimate \cite{Lipeter} thus applies
		and gives
		\[
		\lambda_1(\omega_\kappa)\ge
		\frac{c(n)}{\diam(X,\omega_\kappa)^2}.
		\]
		Combining the two bounds and substituting the formula for $V_\kappa$ yields
		\[
		\diam(X,\omega_\kappa)^2
		\ge\frac{c(n)V_\kappa\sigma_f^2}{2c_f A_\kappa}
		=\frac{c(n)\sigma_f^2}{2c_f}
		\left(1+\frac{B_\kappa}{2A_\kappa}\right)
		\geq\frac{c(n)\sigma_f^2}{2c_f}.
		\]
		This gives another proof of Theorem~\ref{thm:kummer}(b).
	\end{proof}
	
	\begin{proof}[Proof of Theorem~\ref{thm:kummer}\textup{(c)}]
		To prove part (c), let $\kappa_i\in\Kah_X$ satisfy $\kappa_i\to0$, and put
		\[
		\omega_i=\CY(\alpha+\kappa_i),
		\qquad
		D_i=\diam(X,\omega_i).
		\]
		Part (b) gives $D_i\geq c_0$, while the uniform diameter theorem
		\cite[Theorem 4.1]{Tosatti} gives $D_i\leq C$.  Moreover,
		\[
		\Vol(X,D_i^{-2}\omega_i)
		=D_i^{-4}\left(\alpha\cdot\kappa_i+\frac12\kappa_i^2\right)
		\longrightarrow0.
		\]
		By Gromov compactness, the unit diameter metrics $D_i^{-2}\omega_i$ are
		precompact.  Sun--Zhang's classification \cite[Theorem 1.1]{SunZhang} leaves
		limit dimensions one, two, and three, while Theorem
		\ref{thm:dimension-exclusion} excludes dimensions two and three because
		$\alpha+\kappa_i\to\alpha$ and $\inf_iD_i>0$. Thus every cluster point is
		isometric to $[0,1]$.
	\end{proof}

	\begin{remark}
		Along any subsequence for which $D_i\to D_\infty$, the unnormalized limit
		is the interval of length $D_\infty$. Without diameter convergence, Theorem
		\ref{thm:kummer}(c) makes no claim about an unnormalized limit.
	\end{remark}

	\section{New examples of rigid currents}
	In this and next sections, we construct new examples of rigid irrational isotropic nef classes on K3 surfaces.
	
	\subsection{Mechanism of Filip--Tosatti and Sibony--Soldatenkov--Verbitsky}
	First, we review the result of Filip--Tosatti~\cite[Theorem 4.3.1]{FilipTosatti} that motivates the problem.
	Let $X$ be a K3 surface containing no $(-2)$-curves. Under this assumption, one can show that the ample cone $\Amp(X)$ is a
	component of the full positive cone, compare Theorem~\ref{thm:weyl}. Moreover, $\Aut(X)$ acts on its projectivization as a lattice. A homogeneous dynamics argument then proves
	that there is a canonical rigid current in every irrational isotropic boundary class
	\cite[Theorem 4.3.1]{FilipTosatti}, and this canonical rigid current has a continuous potential
	\cite[Theorem 4.2.2]{FilipTosatti}. There are two main ingredients in the proof that every irrational isotropic nef class is rigid~\cite{FilipTosatti}, which we recall below.

	\subsubsection{Semicontinuity of diameter}
	Fix a nonzero holomorphic volume form $\Omega$ on a Calabi--Yau manifold $X$ and define the canonical volume form
	$$
	\mu=\frac{\Omega\wedge\overline\Omega}
	{\int_X\Omega\wedge\overline\Omega}.
	$$
	One important property of this probability measure is that it is invariant under the automorphism group $\Aut(X)$.
	
	The following notion of the diameter of a pseudoeffective class on a K\"ahler manifold was introduced by Sibony--Soldatenkov--Verbitsky~\cite{SSV} in their study of the rigidity of certain nef classes on hyperk\"ahler manifolds. For a pseudoeffective class $\alpha$, let $\cP(\alpha)$ be the weakly
	compact set of closed positive currents in $\alpha$.  After fixing a smooth
	representative $\theta_\alpha$, this is the usual $L^1$-compactness of
	normalized $\theta_\alpha$-psh potentials.
	For $T,T'\in\cP(\alpha)$, write
	$$
	T-T'=dd^c u,\qquad \int_Xu\,d\mu=0,
	$$
	and define
	$$
	d(\alpha):=\sup_{T,T'\in\cP(\alpha)}\int_X|u|\,d\mu.
	$$
	\begin{lemma}\label{thm:ssvdiam} \cite{SSV}
		Let $\alpha$ be a pseudoeffective class. Then the following statements hold:
		\begin{enumerate}
			\item $d(\alpha)=0$ if and only if $\alpha$ is rigid;
			\item
			$d(t\alpha)=t\,d(\alpha)$ for any $t>0$;
			\item 
			$d(f^*\alpha)=d(\alpha)$ for $f\in\Aut(X)$;
			\item 
			$d(\alpha)\ge\limsup_jd(\alpha_j)$ for $\alpha_j\to\alpha$ in $H^{1,1}(X,\R)$.
			
		\end{enumerate}
	\end{lemma}
	\subsubsection{Orbit density in the ample cone}
	Let $X$ be a K3 surface. The standing assumption of Filip--Tosatti is that $X$ contains no $(-2)$-curves. Let $N:=\NS(X)_\R$ be its N\'eron--Severi space, and let $\textbf{G}:=SO(N)$ be the orthogonal group preserving the intersection pairing on $N$. We denote by $G:=\textbf{G}^0(\mathbb R)$ the connected component of the identity, and by
	$$
	\Gamma:=\Aut(X)\cap G
	$$
	the corresponding subgroup. Under the standing assumption that there are no $(-2)$-curves, $\Gamma$ is a lattice in $G$. Rigidity theorems in homogeneous dynamics then imply the following statement.
	
	\begin{lemma}\label{thm:density} \cite[Theorem 7.4.2]{FilipTosatti}
		Let $\alpha\in\partial\Amp(X)$ be a point on the boundary of the ample cone. Then either its orbit $\Aut(X)\cdot\alpha$ is dense in the ample cone, or $\alpha$ is proportional to an integral vector $[E]$ and its orbit is discrete.
	\end{lemma}
	\subsubsection{Sketch of proof}
	With the preceding lemma in hand, we can prove the rigidity of $\alpha$ as follows. We first choose a hyperbolic element
	$g\in\Gamma$ and let $\alpha_+$ be its expanding eigenvector, i.e.,
	$$
	g^*\alpha_+=\lambda\alpha_+,\qquad \lambda>1.
	$$
	Lemma~\ref{thm:ssvdiam} yields that
	$$
	d(\alpha_+)=d(g^*\alpha_+)=d(\lambda\alpha_+)
	=\lambda d(\alpha_+).
	$$
	Thus $d(\alpha_+)=0$.
	Then, by Lemma~\ref{thm:density}, there is a sequence $\gamma_j\in\Aut(X)$ such that
	$$\gamma_j^*\alpha\to\alpha_+.$$
	Again by Lemma~\ref{thm:ssvdiam}, it follows that
	$$
	0=d(\alpha_+)
	\ge\limsup_{j\to\infty}d(\gamma_j^*\alpha)
	=d(\alpha)\ge0.
	$$
	Thus $d(\alpha)=0$ and $\alpha$ is rigid.
	\subsection{New rigid irrational isotropic nef classes: outline}
	In this subsection, inspired by the work of Sibony--Soldatenkov--Verbitsky~\cite{SSV} and Filip--Tosatti~\cite{FilipTosattiSmooth}, we construct new examples of rigid isotropic classes on K3 surfaces, especially in the presence of $(-2)$-curves. Such irrational (but not strongly irrational) rigid isotropic classes appear not to have been known previously. The starting point of our construction is the Torelli theorem for K3 surfaces. Roughly speaking, to construct a K3 surface it suffices to fix a marking in the period domain. To produce an automorphism of a fixed K3 surface, it suffices to produce a Hodge isometry of $\Kthree$ preserving the ample cone. Thus, in addition to the two main ingredients mentioned above from~\cite{FilipTosatti}, our main input for constructing new examples of rigid nef classes is lattice-theoretic and computational.
	
	To this end, we first fix a primitive embedding of a lattice $S\hookrightarrow\Kthree$ and an orthogonal decomposition, with respect to the standard bilinear form of $\Kthree$,
	$$
	S_\R=N_\R\oplus L_\R,
	$$
	where $N$ is a negative-definite ADE lattice and $L$ is a Lorentz lattice. Recall that the signature of $\Kthree$ is $(3,19)$ and that our $S$ will be chosen orthogonal to the period.
	The negative-definite lattice $N$ is generated by the classes of the
	prescribed ADE curves. We choose a K3 surface satisfying $\NS(X)_\R=S_\R$. The main idea is to find a subgroup
	$\Gamma<\Aut(X)$ that fixes $N$ pointwise and acts nontrivially on
	$L=N^\perp$. At the same time, we require the image of $\Gamma$ in $SO(1,\nu-1)$ to be uniform, where $\nu$ is the rank of $L$. Godement's criterion will be useful for obtaining such a group.
	
	Recall that
	a subgroup $\Gamma<G$ of a Lie group is a \emph{lattice} if it is
	discrete and $\Gamma\backslash G$ has finite invariant volume and it is
	\emph{uniform} if this quotient is compact.  An \emph{arithmetic subgroup}
	of $SO(L_\Q)$ is a subgroup commensurable with the integral points for
	some lattice in $L_\Q$. We will need the following compactness criterion for the quotient by $\Gamma$.
	\begin{lemma}\label{lem:godement}\cite{Morris}
		Let $(L_{\mathbb Q},q)$ be a Lorentz lattice defined over $\mathbb Q$. Then the following statements are equivalent:
		\begin{enumerate}
			\item $\Gamma\backslash SO(L_\mathbb R,q)$ is compact;
			\item \(\Gamma\) contains no nontrivial unipotent elements;
			\item 
			\(\operatorname{rank}_{\mathbb Q}SO(L_\mathbb R,q)=0\);
			\item $L_{\mathbb Q}$ has no nonzero rational isotropic vector, i.e.,
			$$
			x\in L_\Q,\quad x^2=0\quad\Longrightarrow\quad x=0.
			$$
		\end{enumerate}
	\end{lemma}
	
	

	\smallskip
	\noindent
	


	Godement's criterion is not enough to establish the density of the orbit $\Gamma\cdot\alpha$ of a nonzero $q$-null vector $\alpha$ in $SO^+(1,n)$. Following the idea of Sibony--Soldatenkov--Verbitsky~\cite{SSV}, we also need the following form of Ratner's density theorem for unipotent flows.
	\begin{lemma}\label{lem:ratner}
		Let $n\ge2$ and let $\Gamma<SO^+(1,n)$ be a uniform lattice. For every
		nonzero future null vector $\alpha$, its orbit $\Gamma\alpha$ is dense
		in the nonzero future null cone.
	\end{lemma}
	
	\begin{proof}
		Write $G=SO^+(1,n)$, fix a nonzero future null vector $\xi_0$, and let
		$P=MAN$ be the parabolic subgroup stabilizing
		$\R_{>0}\xi_0$. Its maximal horospherical subgroup $N$ fixes
		$\xi_0$, and the stabilizer of the vector is $MN$, where
		$M\simeq SO(n-1)$. Consequently, the actual nonzero future null cone is
		$$
		G/(MN),\qquad hMN\longmapsto h\xi_0.
		$$
		By contrast, its projectivization is $G/P$.
		
		Write $\alpha=g\xi_0$. Since $\Gamma\backslash G$ is compact, Ratner's density theorem for unipotent flows~\cite{Ratner91} gives
		$$
		\overline{\Gamma gN}=G.
		$$
		Applying the orbit map $h\mapsto h\xi_0$ and using that $N$ fixes
		$\xi_0$ shows that $\Gamma\alpha$ is dense in $G/(MN)$. The distinction
		between $G/(MN)$ and $G/P$ shows that this is density on the actual cone,
		including the radial scale, rather than only density of rays on its
		projectivization. 
	\end{proof}
	We are now ready to prove Theorem~\ref{thm:newrigid}, which constructs new rigid nef isotropic classes.
	
	\begin{proof}[Proof of Theorem~\ref{thm:newrigid}]
		The proof is a nearly verbatim adaptation of the argument of Filip--Tosatti~\cite{FilipTosattiSmooth}. We only sketch it. Choose a holomorphic volume form $\Omega$ and normalize
		$$
		\mu_X=
		\frac{\Omega\wedge\overline\Omega}
		{\int_X\Omega\wedge\overline\Omega}.
		$$
		Note that $\mu_X$ is automorphism invariant. Let
		$$
		\rho:\Gamma\longrightarrow SO^+(L_\R),\qquad
		\rho(f):=(f^{-1})^*|_{L_\R},
		\qquad \Lambda:=\rho(\Gamma).
		$$
		Choose a hyperbolic element $\lambda\in\Lambda$ and a lift
		$f\in\Gamma$. Let $\alpha_+$ be the expanding eigenvector of $\lambda$. Then $d(\alpha_+)=0$. By Lemma~\ref{lem:ratner} and Lemma~\ref{thm:ssvdiam}, one has
		$$
		0=d(\alpha_+)
		\ge\limsup_{j\to\infty}d(\gamma_j^*\alpha)
		=d(\alpha)\ge0.
		$$
		Thus $d(\alpha)=0$.
	\end{proof}
	Given Theorem~\ref{thm:newrigid}, we now explain how to find appropriate lattices $S,N,L$ satisfying conditions~(A) and~(B) of that theorem.
	
	
	
	
	\smallskip
	\noindent
	\emph{Condition (A):} 
	Intersecting the nef cone with $N_\R^\perp$ imposes equality on the walls
	belonging to the chosen ADE configuration.  This gives a candidate face,
	but an additional root $\delta\in\NS(X)\setminus N$ could, in principle,
	cut its relative interior.  
	Thus, the key to satisfying condition~(A) is to ensure that $\delta^\perp$ does not meet $\C_L$. No additional root wall then cuts the relative interior, and after choosing
	the appropriate K\"ahler chamber we obtain
	\begin{equation*}
		\Nef(X)\cap N_\R^\perp=\overline{\C_L}.
	\end{equation*}

	\smallskip
	\noindent
	\emph{Condition (B):}
	We define the following positive-cone-preserving arithmetic subgroup of $SO^+(L_\R)$:
	$$O^+(L)
	:=\{g\in O(L):g(\C_L)=\C_L\}.
	$$
	The dual lattice and
	the discriminant group of $L$ are
	$$
	L^\vee=\{x\in L_\Q:(x,L)\subset\Z\},
	\qquad
	A_L=L^\vee/L.
	$$
	The finite group $A_L$ carries the discriminant quadratic form $q_L$,
	and every isometry of $L$ acts on $(A_L,q_L)$.  The
	\emph{discriminant kernel} is
	$$
	\Gamma:=O(L)_{\mathrm{disc}}
	=\ker\bigl(O(L)\longrightarrow O(A_L,q_L)\bigr).
	$$
	Since $A_L$ is finite,
	this kernel has finite index in $O(L)$.
	
	We then consider the finite-index subgroup $O^+(L)\cap O(L)_{\mathrm{disc}}$ of $O(L)$. An element $g$ of this subgroup acts
	on $\NS(X)=N\oplus L$ as $\mathrm{id}_N\oplus g$. It acts trivially on
	the discriminant group of $\NS(X)$ and therefore, by the lattice-gluing criterion~\cite{Nikulin}, extends to the K3 lattice by the identity on the
	transcendental lattice.  For the chamber selected in condition $(A)$ above, these extended isometries preserve
	all the root inequalities which define the K\"ahler chamber.  The strong
	Torelli theorem, Theorem~\ref{thm:torelli}, for K3 surfaces then
	realizes them as automorphisms of $X$.

	Thus, the key to satisfying condition~(B) is to ensure that the subgroup $\Gamma\cap SO^{+}(L_\R)$ is a uniform lattice
	in $SO^+(L_\R)$.
	

	\section{The K3 lattice realization}
	\label{sec:k3-lattice-realization}
	
	In this section, we construct appropriate lattices that yield the required nef faces and arithmetic
	automorphism groups satisfying the assumptions of Theorem~\ref{thm:newrigid}. Thus, we obtain many new rigid irrational isotropic nef classes on K3 surfaces containing $(-2)$-curves.
	We remark that the assumptions on the lattices in this section are very mild and were optimized with the assistance of AI based on the mechanism of Theorem~\ref{thm:newrigid}.
	
	Let us start with an anisotropic lattice $L$ over $\mathbb{Q}$. For $\nu\in\{3,4\}$, put
	$$
	L_{p,\nu}=\langle4p\rangle\oplus
	\langle-4\rangle^{\oplus(\nu-1)}.
	$$
	More explicitly,
	$$
	(x_0,\ldots,x_{\nu-1})^2
	=4\left(px_0^2-\sum_{j=1}^{\nu-1}x_j^2\right).
	$$
	Then $L_{p,\nu,\R}$ is Lorentzian, with signature $(1,\nu-1)$.
	
	\begin{lemma}\label{lem:anisotropy}
		Let $\nu\in\{3,4\}$ and $p>0$. Then the following are equivalent:
		\begin{enumerate}
			\item $L_{p,\nu,\Q}$ is anisotropic.
			\item $p$ is not a sum of $\nu-1$ rational squares.
		\end{enumerate}
		If $p\equiv7\pmod8$, these conditions hold simultaneously for
		$\nu=3$ and $\nu=4$. For $\nu=3$, they are also equivalent to the
		existence of a prime $q\equiv3\pmod4$ which occurs to odd order in $p$.
	\end{lemma}
	
	\begin{proof}
		An isotropic vector of $L_{p,\nu,\Q}$ is a nonzero rational solution of
		$$
		px_0^2=x_1^2+\cdots+x_{\nu-1}^2.
		$$
		If $x_0=0$, then every $x_j=0$, so any nonzero solution has
		$x_0\ne0$ and is equivalent to a representation of $p$ as a sum of
		$\nu-1$ rational squares.
		
		Suppose now that $p\equiv7\pmod8$ and clear denominators to obtain a
		primitive integral solution. If $x_0$ is odd, then for $\nu=3$ the
		equation is impossible modulo $4$, while for $\nu=4$ it is impossible
		modulo $8$: a sum of two squares is never $3$ modulo $4$, and a sum of
		three squares is never $7$ modulo $8$. If $x_0$ is even, the right hand
		side is divisible by $4$. A sum of at most three integral squares is
		divisible by $4$ only when each summand is even. Thus all $x_j$ are
		even, contradicting primitivity.
		
		For $\nu=3$, the final equivalence is the classical two-squares
		criterion, applied after clearing denominators, multiplication by a
		rational square changes every prime valuation by an even integer.
	\end{proof}
	
	The following projection lemma is the main lattice-theoretic observation
	in the K3 construction.
	
	\begin{lemma}\label{lem:projection}
		Let $N$ be an ADE lattice, let $\nu\in\{3,4\}$, and suppose
		$L_{p,\nu,\Q}$ is anisotropic. If
		$\delta=n+\ell\in N\oplus L_{p,\nu}$ satisfies $\delta^2=-2$, then exactly one
		of the following occurs:
		\begin{enumerate}
			\item $\ell=0$ and $n$ is a root of $N$;
			\item $\ell^2>0$.
		\end{enumerate}
	\end{lemma}
	
	\begin{proof}
		The root equation gives
		$$
		\ell^2=-2-n^2.
		$$
		If $n=0$, it would give $\ell^2=-2$, impossible because all norms in
		$L_{p,\nu}$ are divisible by $4$. If $n^2=-2$, anisotropy gives
		$\ell=0$. Every other nonzero vector in the negative-definite even root
		lattice $N$ has $n^2\le-4$, and then $\ell^2=-2-n^2>0$.
	\end{proof}
	
	We now apply the K3 Weyl-group description to choose a K\"ahler chamber
	whose closure contains the complete Lorentzian face.
	
	\begin{proposition}\label{prop:chamber}
		Let $\nu\in\{3,4\}$ and let $S=N\oplus L_{p,\nu}$, with
		$L_{p,\nu,\Q}$ anisotropic. Suppose that a
		projective K3 surface $X$ has $\NS(X)\simeq S$.  After changing the
		marking by the Weyl group, one may assume that
		\begin{enumerate}
			\item the simple roots of $N$ are represented by an ADE configuration of
			smooth rational curves;
			\item for one component $\C_{p,\nu}$ of
			$\{x\in L_{p,\nu,\R}:x^2>0\}$, one has
			$$
			\overline{\C_{p,\nu}}=\Nef(X)\cap N^\perp.
			$$
		\end{enumerate}
	\end{proposition}
	
	\begin{proof}
		We first choose a system $\Delta_N$ of simple roots and a vector
		$v\in N_\R$ such that $v\cdot a>0$ for every
		$a\in\Delta_N$. Fix
		$h_0\in\C_{p,\nu}$ with $h_0^2=1$.  We claim that
		$$
		k_\varepsilon=h_0+\varepsilon v
		$$
		belongs to a root chamber for all sufficiently small $\varepsilon>0$.
		After decreasing $\varepsilon$ if necessary, one also has
		$k_\varepsilon^2=1+\varepsilon^2v^2>0$.
		
		For roots in $N$, choose positivity using the Weyl chamber determined by
		$v$. If $\delta=n+\ell$ is a root not in $N$, choose its sign so that
		$\ell\in\C_{p,\nu}$. By Lemma~\ref{lem:projection} and the Lorentzian reverse Cauchy--Schwarz inequality,
		$$
		h_0\cdot\ell\ge\sqrt{\ell^2}=\sqrt{-2-n^2}.
		$$
		Since $N_\R$ is negative definite, there is $C_v>0$ such that
		$\lvert v\cdot n\rvert\le C_v\sqrt{-n^2}$. For $n^2\le-4$,
		$$
		\frac{\sqrt{-2-n^2}}{\sqrt{-n^2}}\ge\frac1{\sqrt2}.
		$$
		Thus $k_\varepsilon\cdot\delta>0$ for every such positive root whenever
		$\varepsilon<(\sqrt2 C_v)^{-1}$. The same $k_\varepsilon$ is positive
		on all positive roots of $N$. It therefore lies in a chamber of the
		positive cone.  By the K3 Weyl-group and cone description in
		Theorem~\ref{thm:weyl}, changing the marking by a unique Weyl-group element
		makes this chamber the marked K\"ahler chamber. Since $k_\varepsilon$ is then a K\"ahler class and
		$k_\varepsilon\to h_0$, the class $h_0$ is nef.

		We finally let $h\in\C_{p,\nu}$. It is orthogonal to every root in $N$, while
		for every other positive root $\delta=n+\ell$ one has
		$$
		h\cdot\delta=h\cdot\ell>0.
		$$
		Hence $h$ lies in the closure of the chosen K\"ahler chamber.  Taking the
		closure in $L_{p,\nu,\R}$ gives
		$\overline{\C_{p,\nu}}\subset\Nef(X)\cap N^\perp$. Conversely,
		$N^\perp\cap\NS(X)_\R=L_{p,\nu,\R}$, and every nef class lies in the closure
		of the positive-cone component containing the K\"ahler cone.  Therefore
		$\Nef(X)\cap N^\perp=\overline{\C_{p,\nu}}$.
	\end{proof}

	Recall that we are interested in the following two groups.
	$$
	\begin{aligned}
		O^+(L_{p,\nu})
		&:=\{g\in O(L_{p,\nu}):g(\C_{p,\nu})=\C_{p,\nu}\},\\
		SO^+(L_{p,\nu,\R})
		&:=\{g\in SO(L_{p,\nu},\mathbb R):g(\C_{p,\nu})=\C_{p,\nu}\}.
	\end{aligned}
	$$
	
	\begin{proposition}\label{prop:aut}
		In the setting of Proposition~\ref{prop:chamber}, assume that the embedding
		$S\hookrightarrow\Kthree$ is primitive and that the period $\sigma$ is chosen such that the induced K3 surface $X$ satisfies $\NS(X)=S$. Then the image of
		\begin{equation}\label{eqn:auto}
			\Aut(X,N)_{\mathrm{pt}}
			:=\{f\in\Aut(X):f^*|_N=\mathrm{id}_N\}
			\longrightarrow O^+(L_{p,\nu})
		\end{equation}
		contains the finite-index subgroup
		$$
		\Gamma_{p,\nu}:=\ker\bigl(O^+(L_{p,\nu})
		\longrightarrow O(A_{L_{p,\nu}})\bigr).
		$$
		Moreover,
		$\Gamma_{p,\nu}^1:=\Gamma_{p,\nu}\cap SO^+(L_{p,\nu,\R})$ has finite
		index in $\Gamma_{p,\nu}$ and is a uniform lattice in
		$SO^+(L_{p,\nu,\R})\simeq SO^+(1,\nu-1)$.
	\end{proposition}
	
	\begin{proof}
		It is clear that the group $\Gamma_{p,\nu}$ has finite index in
		$O^+(L_{p,\nu})$. We next show that any
		$g\in\Gamma_{p,\nu}$ is contained in the image of the group homomorphism~\eqref{eqn:auto}. The isometry
		$\mathrm{id}_N\oplus g$ acts trivially on
		$A_S=A_N\oplus A_{L_{p,\nu}}$. Since $S\subset\Kthree$ is primitive, it
		extends by the identity on $S^\perp$ to an integral isometry of
		$\Kthree$ \cite[Corollary 1.5.2]{Nikulin}.  It is a Hodge isometry for
		the chosen period.
		
		The extended isometry fixes the positive roots of $N$ and also preserves
		$\C_{p,\nu}$. By Lemma~\ref{lem:projection}, it therefore permutes the positive roots not contained in the ADE root lattice $N$. Recall that roots are integral elements of self-intersection $-2$. It follows that the isometry preserves the K\"ahler chamber and hence the K\"ahler cone of $X$. By the strong global Torelli theorem, Theorem~\ref{thm:torelli}, there is a unique automorphism
		$f_g\in\Aut(X)$ whose pullback on $H^2(X,\Z)$ is this extended Hodge
		isometry. Thus $f_g^*|_N=\mathrm{id}_N$ and
		$f_g^*|_{L_{p,\nu}}=g$, so every $g\in\Gamma_{p,\nu}$ belongs to the displayed
		image.  Moreover, uniqueness gives $f_{gh}=f_h\circ f_g$, because
		pullback reverses composition. Hence the map $g\mapsto f_{g^{-1}}$ is a homomorphism.
		
		It remains to prove cocompactness of the subgroup $\Gamma^1_{p,\nu}$ of $SO^{+}(L_{p,\nu,\R})$. The rational quadratic space
		$L_{p,\nu,\Q}$ is anisotropic, so the
		$\Q$-rank of the algebraic group $SO(L_{p,\nu,\Q})$ is zero. The group
		$\Gamma_{p,\nu}^1$ is arithmetic in $SO^+(L_{p,\nu,\R})$. Godement's compactness
		criterion \cite{Morris} therefore implies that
		$\Gamma_{p,\nu}^1\backslash\mathbb P(\C_{p,\nu})$ is compact. Thus
		$\Gamma_{p,\nu}^1$ has finite index in $\Gamma_{p,\nu}$ and is a uniform
		lattice in
		$SO^+(L_{p,\nu,\R})\simeq SO^+(1,\nu-1)$. Passing to a torsion-free
		finite-index subgroup does not change this conclusion.
	\end{proof}
	
	We now give a convenient numerical condition for the
	primitive embedding required in Theorem~\ref{thm:newrigid}.
	The following proposition is the form which will be used in the examples.
	
	\begin{proposition}\label{prop:embedding}
		Let $N$ be an ADE lattice and let $\nu\in\{3,4\}$. Set
		$$
		S_{N,p,\nu}:=N\oplus L_{p,\nu}.
		$$
		If
		\begin{equation}\label{eq:stable}
			\rk(N)+\length(A_N)\le20-2\nu,
		\end{equation}
		then, for every $p>0$, the lattice $S_{N,p,\nu}$ admits a primitive
		embedding into $\Kthree$.
	\end{proposition}
	
	\begin{proof}
		The lattice $S_{N,p,\nu}$ is even of signature
		$(1,\rk(N)+\nu-1)$, and
		$$
		\length(A_{S_{N,p,\nu}})
		\le \length(A_N)+\length(A_{L_{p,\nu}})
		\le \length(A_N)+\nu.
		$$
		Its orthogonal complement in $\Kthree$ would have signature
		$(2,20-\rk(N)-\nu)$ and rank $22-\rk(N)-\nu$. Condition
		\eqref{eq:stable} implies
		$$
		22-\rk(N)-\nu\ge\length(A_N)+\nu+2
		\ge\length(A_{S_{N,p,\nu}})+2.
		$$
		The strict stable-range form of Nikulin's primitive embedding theorem
		\cite[Theorem 1.12.2 and Corollary 1.12.3]{Nikulin} therefore applies.
		The signature inequalities are automatic: the complement has two positive
		directions and at least three negative directions under
		\eqref{eq:stable}.
	\end{proof}
	
	\begin{remark}
		The inequality~\eqref{eq:stable} is sufficient, but not necessary, for the existence of a primitive embedding.
	\end{remark}

	\subsection{Numerical examples}
	In this subsection, we give explicit numerical examples. Every example is
	obtained from the same facewise construction with
	$S_{N,p,\nu}=N\oplus L_{p,\nu}$. The
	tables only record lattice data for which the K3
	realization is guaranteed. The rigidity and regularity arguments are the
	same for all of them.
	
	\subsubsection{Parameters \texorpdfstring{$p$}{p}}
	
	For $\nu=3$, Lemma~\ref{lem:anisotropy} shows that every $p$ having a
	prime $q\equiv3\pmod4$ to odd exponent produces an example. 
	For either $\nu=3$ or $\nu=4$, the common subfamily
	$$
	p\equiv7\pmod8
	$$
	is anisotropic. For fixed $(N,\nu)$, distinct $p$ give pairwise
	nonisometric
	N\'eron--Severi lattices because
	$$
	|\det S_{N,p,\nu}|=4^\nu p\,|\det N|.
	$$
	The moduli space of $S_{N,p,\nu}$-polarized K3 surfaces has dimension
	$$
	20-\rk(S_{N,p,\nu})=20-\nu-\rk(N).
	$$
	
	The stable-range inequality \eqref{eq:stable} becomes
	$$
	\begin{array}{c|c|c|c}
		\nu&\operatorname{sign}(L_{p,\nu})&
		\text{boundary of rays}&
		\text{sufficient condition}\\ \hline
		3&(1,2)&S^1&\rk(N)+\length(A_N)\le14\\
		4&(1,3)&S^2&\rk(N)+\length(A_N)\le12.
	\end{array}
	$$
	The detailed catalogues below use $\nu=3$, which has the larger
	stable range. Replacing $14$ by $12$ gives the corresponding rank-four
	catalogue without changing any proof.
	
	\subsubsection{All irreducible ADE types guaranteed by the stable range}
	
	For an irreducible root lattice, the discriminant lengths are
	$$
	\begin{array}{c|ccccc}
		N&A_n&D_{2m}&D_{2m+1}&E_6, E_7&E_8\\ \hline
		\length(A_N)&1&2&1&1&0.
	\end{array}
	$$
	Thus Proposition~\ref{prop:embedding}, with $\nu=3$, gives:
	$$
	\begin{array}{c|c|c}
		N&\text{guaranteed types}&\text{dimension of the K3 family}\\ \hline
		A_n&1\le n\le13&17-n\\
		D_n,\ n\ {\rm odd}&5\le n\le13&17-n\\
		D_n,\ n\ {\rm even}&4\le n\le12&17-n\\
		E_6&E_6&11\\
		E_7&E_7&10\\
		E_8&E_8&9
	\end{array}
	$$
	Each row represents infinitely many N\'eron--Severi lattices as $p$
	varies, and every individual rank-three surface carries a full circle
	$\mathbb P(\partial\C_p)$ of rigid classes.
	
	\subsubsection{Uniform repeated-component series}
	
	Using the conservative additive bound for discriminant length, one obtains
	$$
	\begin{array}{c|c}
		N&\text{guaranteed multiplicities}\\ \hline
		A_1^{\oplus k}&1\le k\le7\\
		A_2^{\oplus k}&1\le k\le4\\
		A_3^{\oplus k}&1\le k\le3\\
		A_4^{\oplus k}&1\le k\le2\\
		A_5^{\oplus k}&1\le k\le2\\
		A_6^{\oplus k}&1\le k\le2\\
		D_4^{\oplus k}&1\le k\le2\\
		D_5^{\oplus k}&1\le k\le2\\
		E_6^{\oplus k}&1\le k\le2.
	\end{array}
	$$
	
	\subsubsection{Complete two-component catalogue under the conservative test}
	
	The following table records all unordered two-component ADE sums satisfying
	$\rk(N)+b(N)\le14$, where $b$ is the sum of the individual
	discriminant-length bounds.  Since
	$\length(A_N)\le b(N)$, every entry satisfies \eqref{eq:stable} for
	$\nu=3$.
	
	\begin{longtable}{>{\raggedright\arraybackslash}p{0.19\textwidth}
			>{\raggedright\arraybackslash}p{0.74\textwidth}}
		\toprule
		\text{first factor}&\text{allowed second factors}\\
		\midrule
		\endhead
		$A_1$&
		$A_1,\ldots,A_{11};\ D_4,\ldots,D_{11};\ E_6,E_7,E_8$\\
		$A_2$&
		$A_2,\ldots,A_{10};\ D_4,\ldots,D_9;\ E_6,E_7,E_8$\\
		$A_3$&
		$A_3,\ldots,A_9;\ D_4,\ldots,D_9;\ E_6,E_7,E_8$\\
		$A_4$&
		$A_4,\ldots,A_8;\ D_4,\ldots,D_7;\ E_6,E_7,E_8$\\
		$A_5$&
		$A_5,A_6,A_7;\ D_4,D_5,D_6,D_7;\ E_6,E_7,E_8$\\
		$A_6$&
		$A_6;\ D_4,D_5;\ E_6$\\
		$A_7$&
		$D_4,D_5$\\
		$D_4$&
		$D_4,D_5,D_6,D_7;\ E_6,E_7,E_8$\\
		$D_5$&
		$D_5,D_6,D_7;\ E_6,E_7,E_8$\\
		$E_6$&
		$E_6$\\
		\bottomrule
	\end{longtable}
	
	This catalogue contains $98$ unordered pairs.  Representative high-rank
	examples include
	$$
	\begin{gathered}
		A_1\oplus A_{11},\quad A_2\oplus A_{10},\quad
		A_3\oplus A_9,\quad A_4\oplus A_8,\quad A_5\oplus A_7,\\
		A_5\oplus E_8,\quad D_5\oplus E_8,\quad
		D_5\oplus D_7,\quad E_6\oplus E_6 .
	\end{gathered}
	$$
	
	\subsubsection{Examples with three or more components}
	
	There are many further examples.  The following list illustrates the range,
	each entry satisfies the conservative inequality
	$\rk(N)+b(N)\le14$:
	$$
	\begin{array}{lll}
		A_1^7,&A_1^5\oplus A_2,&A_1^4\oplus A_3,\\
		A_1^3\oplus A_2^2,&A_1^2\oplus A_2^3,&A_2^4,\\
		A_1^3\oplus D_4,&A_1^2\oplus D_5,&A_1\oplus A_2\oplus D_5,\\
		A_1^2\oplus E_6,&A_1\oplus A_2\oplus E_6,&
		A_1\oplus A_3\oplus E_6,\\
		A_1^2\oplus E_7,&A_1\oplus A_2\oplus E_7,&
		A_1^3\oplus E_8,\\
		A_1\oplus A_2\oplus E_8,&A_2^2\oplus A_3,&
		A_1\oplus A_3^2 .
	\end{array}
	$$
	Together with all admissible $p$, even this short list yields infinitely
	many pairwise nonisometric projective K3 families.

	\bibliographystyle{plain}
	\bibliography{rigidity_gh_k3_draft}
	
\end{document}